\documentclass[10pt]{amsart}
\usepackage{amsmath}
\usepackage{amsthm}
\usepackage{enumerate}
\usepackage{amssymb}
\usepackage[utf8]{inputenc}
\usepackage[colorlinks  = true,
            citecolor   = blue,
            linkcolor   = blue,
            bookmarks   = true,
            linktocpage = true]{hyperref}
\usepackage[all]{xy}
\usepackage{tikz-cd}

\title{An Equivariant Theory for the Bivariant Cuntz Semigroup}

\author{Gabriele N. Tornetta}
\address{School of Mathematics and Statistics, University of Glasgow, 15 University Gardens, G12 8QW, Glasgow, UK}
\email{g.tornetta.1@research.gla.ac.uk}

\thanks{\emph{Supported by:}  EPSRC Grant EP/I019227/2}

\date{\today}

\newcommand{\A}{}

\newcommand{\Cs}{C$^*$}

\newcommand{\M}{{\mathcal M}}

\newcommand{\bb}[1]{\mathbb{#1}}
\newcommand{\IN}{\bb N}

\newcommand{\IR}{\bb R}
\newcommand{\CC}{\bb C}

\newcommand{\Ad}{\operatorname{Ad}}
\newcommand{\dd}{^{**}}
\newcommand{\de}{{\operatorname d}}

\newcommand{\id}{\operatorname{id}}

\newcommand{\hoplus}{\operatorname{{\hat{\oplus}}}}
\newcommand{\SOT}{\text{\normalfont\scshape sot}}
\newcommand{\supp}{\operatorname{supp}}

\newcommand{\Mf}{\operatorname{Mf}}

\newcommand{\cc}{\subset\!\!\!\subset}

\newcommand{\gpze}{\sim_{G,\text{PZ}}}

\newcommand{\cat}[1]{\text{\normalfont\sffamily #1}}

\newcommand{\mor}[1]{\operatorname{#1}}

\newcommand{\Aut}{\mor{Aut}}

\newcommand{\Cu}{\cat{Cu}}

\newcommand{\seq}[1]{\left\{{#1}_n\right\}_{n\in\IN}}
\newcommand{\norm}[1]{\left\Vert{#1}\right\Vert}

\newcommand{\eps}{\epsilon}

\newcommand{\ess}{\text{\upshape ess}}

\theoremstyle{plain}
\newtheorem{lemma}{Lemma}[section]
\newtheorem{theorem}[lemma]{Theorem}
\newtheorem{corollary}[lemma]{Corollary}
\newtheorem{proposition}[lemma]{Proposition}
\newtheorem*{proposition*}{Proposition}
\newtheorem*{theorem*}{Theorem}
\newtheorem*{definition*}{Definition}
\newtheorem*{claim*}{Claim}

\theoremstyle{definition}
\newtheorem{definition}[lemma]{Definition}
\newtheorem{example}[lemma]{Example}

\theoremstyle{remark}
\newtheorem{remark}[lemma]{Remark}

\begin{document}

	\maketitle
	\begin{abstract}
		We provide an equivariant extension of the bivariant Cuntz semigroup introduced in previous work for the case of compact group actions over \Cs-algebras. Its functoriality properties are explored and some well-known classification results are retrieved. Connections with crossed products are investigated, and a concrete presentation of equivariant Cuntz homology is provided. The theory that is here developed can be used to define the equivariant Cuntz semigroup. We show that the object thus obtained coincides with the recently proposed one by Gardella and Santiago, and we complement their work by providing an open projection picture of it.

	\end{abstract}


	\section{Introduction}

During the last couple of decades, the Cuntz semigroup, an invariant for C*-algebras originally proposed by Cuntz \cite{cuntz78} in the late seventies of the previous century, has gained a prominent r\^ole in the Classification Programme of \Cs-algebras initiated by George Elliott. Important contributions came from the results of R\o rdam \cite{rordam03} and Toms \cite{toms}, which brought to light the fact that there are non-isomorphic \Cs-algebras with the same Elliott invariant and that can be told apart by other invariants, like the real rank and their Cuntz semigroup data.

The investigation into the theory of completely positive maps with the order zero property (from now on, c.p.~order zero maps for short) undertaken by Winter and Zacharias \cite{wz}, building on top of previous work on orthogonality preserving maps by Wolff \cite{wolff}, has led to the discovery of deep connections between this special class of maps between \Cs-algebras and the Cuntz semigroup. In particular, they have shown that every such map induces a semigroup homomorphism between the associated Cuntz semigroups. This has opened the doors for a bivariant extension of the theory of the Cuntz semigroup, that has been established in \cite{bcs}. The main idea behind the just cited work is to provide a new framework which is, in spirit, similar to Kasparov's bivariant K-theory, i.e. KK-theory, but that it is reminiscent of the idea of Cuntz comparison of positive structures (in this case c.p.c.~order zero maps). The classification result for the class of unital and stably finite \Cs-algebras, obtained by the author in his Ph.D.~work, and contributed to \cite{bcs}, can be regarded as an analogue of the Kirchberg-Phillips classification of purely infinite \Cs-algebras by KK-theory. Thus, the bivariant Cuntz semigroup complements the classification of purely infinite \Cs-algebras in the sense that it gives a suitable Cuntz-analogue of the KK-theoretic Kirchberg-Phillips classification result for stably finite \Cs-algebras.

The main subject of the present paper, which also stems from part of the Ph.D.~work of the author, is the introduction of an equivariant extension of the bivariant Cuntz semigroup, with the final aim to provide a novel tool for the problem of classification of \Cs-dynamical systems. The theory proposed in this paper is specifically designed for \emph{compact} group actions on \Cs-algebras. With such a tool at one's disposal, the definition of an equivariant extension of the ordinary Cuntz semigroup appears as a natural consequence, since it has been shown in \cite{bcs} that the ordinary Cuntz semigroup can be recovered from the bivariant theory there defined. A notion of equivariant Cuntz semigroup, however, has recently appeared in the work of Gardella and Santiago \cite{gs}, where it is also used to provide a classification result of certain actions on certain \Cs-algebras. Indeed, we shall see that the equivariant theory developed in this paper overlaps with the new definitions that have appeared in \cite{gs}, and that the classification result just mentioned can also be \emph{recovered} within the framework of the present work. 


\subsection*{Outline}

The present paper is organised as follows. In Section \ref{sec:prelim} we recall the definition of the bivariant Cuntz semigroup as presented in \cite{bcs} together with the main classification result for unital and stably finite \Cs-algebras, and also set the notation that will be used throughout. We also provide the basics of equivariant K-theory and KK-theory for the reader's convenience, and references to more self-contained sources are also provided.

In Section \ref{sec:equibcs} we introduce the equivariant theory of the bivariant Cuntz semigroup and we show that most of the properties of the ordinary bivariant Cuntz semigroup of \cite{bcs} carry over to the equivariant setting. Here we restrict our attention to \emph{separable} \Cs-dynamical systems, i.e. those for which the underlying \Cs-algebra is separable and the group acting on it is second countable. We also investigate the relations with crossed products in order to strengthen the analogy with equivariant KK-theory. In particular, we show that a Green-Julg-type theorem holds for the equivariant theory of the bivariant Cuntz semigroup developed in this paper (Theorem \ref{th:jG}). This section is concluded with an equivariant extension of the \emph{Cuntz homology} for compact Hausdorff spaces introduced in \cite{bcs}.

Section \ref{sec:equics} is dedicated to the \emph{equivariant Cuntz semigroup}. This new object emerges as the special case $\Cu^G(\CC, A)$, in complete analogy with the way one can recover the ordinary Cuntz semigroup from the bivariant theory defined in \cite{bcs}. We show that the object thus defined coincides with that introduced in \cite{gs}, and we complement the latter by providing an equivariant extension of the open projection picture of the Cuntz semigroup investigated in \cite{ORT}.

In Section \ref{sec:classification} we show how to use the equivariant bivariant Cuntz semigroup for the problem of classification of actions. This is done by generalising the notion of strict equivalence of \cite{bcs} to the equivariant setting and extending the core result of the paper to the equivariant version given in Theorem \ref{thm:Gsinvtoisowithscale}. In particular, we show how to use the theory developed in this paper to retrieve the classification result of Handelman and Rossmann \cite{hr} on locally representable actions by compact groups on AF algebras (Corollary \ref{cor:bhr}), and the more recent result of Gardella and Santiago of locally representable finite Abelian group actions on inductive limits of one-dimensional non-commutative CW complexes (Corollary \ref{cor:bgs}).


\subsection*{Acknowledgements}

The author wishes to thank his supervisor, Joachim Zacharias, for proposing the investigation of an equivariant theory for the bivariant Cuntz semigroup, and Eusebio Gardella for many fruitful conversations about the equivariant Cuntz semigroup he has developed together with Luis Santiago, as well as many insightful suggestions and comments on an early draft of this paper. We acknowledge the support of EPSRC, grant EP/I019227/2.

	\section{\label{sec:prelim}Preliminaries}

We use this section to collect some preliminary notions and results, and to set most of the notation employed in this paper. In what follows, we shall make use of capital letters $A,B,C,\ldots$ to denote \Cs-algebras, and the notation $A_+$ to denote the cone of the positive elements of $A$. The multiplier algebra of $A$ is denoted by $\M(A)$, while the letter $K$ is used to denote the \Cs-algebra of compact operators on an infinite-dimensional separable Hilbert space. For $k\in\IN$, we shall denote by $M_k$ the complex $k\times k$ matrix algebra, and by $M_k(A)$ the $A$-valued $k\times k$ matrix algebra. With this notation we then have $M_k = M_k(\CC)$.

The Greek letters $\pi,\omega,\ldots$ are used to denote $*$-homomorphisms between \Cs-algebras, while $\phi,\psi$ are reserved for c.p.c.~order zero maps, whose definition is recalled in this section for the reader's convenience.

The letters $E,F$ are reserved for Hilbert modules, while $V,W$ will usually denote Hilbert spaces. Unless otherwise specified, every Hilbert module is a \emph{right} Hilbert module, and every Hilbert $A$-$B$ bimodule is a right Hilbert $B$ module with a left $A$-action. The set of all bounded and adjointable operators on a Hilbert module $E$ will be denoted by $B(E)$. We refer the reader to \cite{mt} and \cite{lance} for an overview of Hilbert (bi)modules.

In describing equivariant theories, we shall reserve the capital letter $G$ for denoting a topological group, usually assumed to be compact and second countable, unless otherwise specified.


\subsection{The Bivariant Cuntz Semigroup}

We recall the main definitions and results of \cite{bcs}, where a bivariant extension of the ordinary theory of the Cuntz semigroup has been proposed. The key ingredient is a notion of comparison between completely positive maps with the order zero property. Such maps were originally introduced in \cite{wintercd}, and subsequently thoroughly investigated in \cite{wz}.

\begin{definition} Let $A$ and $B$ be \Cs-algebras. A completely positive map $\phi:A\to B$ has the order zero property if $\phi(a)\phi(b) = 0$ whenever $ab = 0$, with $a,b\in A_+$.
\end{definition}

The structure of completely positive maps with the order zero property has been established in \cite{wz}, where the authors have built on top of previous work by Wolff \cite{wolff} on orthogonality preserving maps.

\begin{theorem}[Winter-Zacharias]\label{thm:wz} Let $A$ and $B$ be \Cs-algebras and let $\phi:A\to B$ be a completely positive map with the order zero property. Denote by $C_\phi$ the \Cs-algebra generated by the image of $\phi$, that is $C_\phi:= C^*(\phi(A))$. There are a positive element $h_\phi\in\M(C_\phi)_+\cap C_\phi'$ and a non-degenerate $*$-homomorphism $\pi_\phi : A\to\M(C_\phi)\cap\{h_\phi\}'$ such that 
$\norm\phi = \norm{h_\phi}$ and $\phi(a) = h_\phi\pi_\phi(a)$ for any $a\in A$.
\end{theorem}

\begin{definition} Let $A$ and $B$ be \Cs-algebras, and let $\phi,\psi:A\to B$ be c.p.c.~order zero maps. We say that $\phi$ is (Cuntz-)subequivalent to $\psi$, in symbols $\phi\precsim\psi$, if there exists a sequence $\seq b\subset B$ such that
	$$\lim_{n\to\infty}\norm{b_n\psi(a)b_n^*- \phi(a)} \to 0,$$
for all $a\in A$.
\end{definition}

Two c.p.c.~order zero maps $\phi,\psi:A\to B$ are said to be (Cuntz-)equivalent, in symbols $\phi\sim\psi$, if both $\phi\precsim\psi$ and $\psi\precsim\phi$ hold.

For any \Cs-algebra $A$, we denote by $\Delta_A$, or simply by $\Delta$ when no confusion arises, the \emph{diagonal map} on $A$ with values in $A\oplus A$, that is
	$$\Delta(a) := (a, a)$$
for any $a\in A$. Hence, for the direct sum of two c.p.c.~order zero maps $\phi,\psi:A\to B$ we define the binary operation $\hoplus$ as
	$$\phi\hoplus\psi := (\phi\oplus\psi)\circ\Delta.$$
This yields a c.p.c.~order zero map from $A$ to $B\oplus B\subset M_2(B)$, as one can easily verify from the above definition.

\begin{definition}\label{def:bcs} Let $A$ and $B$ be \Cs-algebras. The bivariant Cuntz semigroup $\Cu(A,B)$ is the set of Cuntz-equivalence classes of c.p.c.~order zero maps from $A\otimes K$ to $B\otimes K$, equipped with the binary operation $+:\Cu(A,B)\times\Cu(A,B)\to\Cu(A,B)$ given by
	$$[\phi]+[\psi] := [\phi\hoplus\psi].$$
\end{definition}

It is shown in \cite{bcs} that the above definition yields a monoid, which can also be equipped with a partial order relation $\leq$ to become a partially ordered Abelian monoid.

The major target behind the idea of a bivariant theory for the Cuntz semigroup developed in \cite{bcs} was that of providing a new tool for the classification problem for \Cs-algebras. In this respect, the theory there developed has been proven to be successful, since it provides a tool that is capable of classifying all unital and stably finite \Cs-algebras in a way that formally resembles the Kirchberg-Phillips classification result for the purely infinite case by means of bivariant K-theory, i.e. KK-theory. The fundamental notion is that of strict invertibility, which captures the \emph{right} scale condition that allows lifting a class from $\Cu(A,B)$ to a $*$-isomorphism between $A$ and $B$.

\begin{definition} Let $A$ and $B$ be \Cs-algebras. An element $\Phi\in\Cu(A,B)$ is said to be strictly invertible if there exist c.p.c.~order zero maps $\phi:A\to B$ and $\psi:B\to A$ such that $\Phi = [\phi\otimes\id_K]$, $\psi\circ\phi\sim\id_A$ and $\phi\circ\psi\sim\id_B$.
\end{definition}

It is clear that every $*$-isomorphism $\pi:A\to B$ between two \Cs-algebras induces a strictly invertible element in $\Cu(A,B)$, which is given by the class of the map $\pi\otimes\id_K$. With the above definition, the main classification result of \cite{bcs} can be stated in the following terms.

\begin{theorem} Let $A$ and $B$ be unital and stably finite \Cs-algebras. There is a strictly invertible element in $\Cu(A,B)$ if and only if $A$ and $B$ are $*$-isomorphic.
\end{theorem}

An analogue of Kasparov's intersection product can be defined for the bivariant Cuntz semigroup as well, as discussed in \cite{bcs}. There it is defined simply as \emph{composition} of c.p.c.~order zero map. Specifically, for $[\phi]\in\Cu(A,B)$ and $[\psi]\in\Cu(B,C)$, one defines the product
	$$[\phi]\cdot[\psi]$$
to be the element of $\Cu(A,B)$ given by $[\psi\circ\phi]$. This way, for every \Cs-algebra $A$, the semigroup $\Cu(A,A)$ comes equipped with a natural structure of a semiring, with unit given by $\iota_A := [\id_{A\otimes K}]$. One can then define \emph{invertible elements} as those classes $\Phi\in\Cu(A,B)$ for which there is $\Psi\in\Cu(B,A)$ such that $\Phi\cdot\Psi = \iota_A$ and $\Psi\cdot\Phi = \iota_B$. By defining the \emph{scale} of the bivariant Cuntz semigroup as
	$$\Sigma(\Cu(A,B)):=\{[\phi\otimes\id_K]\in\Cu(A,B)\ |\ \phi:A\to B\text{ is c.p.c.~order zero}\}$$
one can easily verify that the strictly invertible elements of $\Cu(A,B)$ are precisely the invertible elements in the scale that have inverse in $\Sigma(\Cu(B,A))$. Furthermore, if $A$ and $B$ are unital \Cs-algebras, and $\Phi\in\Cu(A,B)$ is strictly invertible, then
	$$\Phi\cdot[1_A] = [1_B],$$
where $[1_A]$ and $[1_B]$ denote the classes of the units of $A$ and $B$ respectively in $\Cu(A)\cong\Cu(\CC,A)$ and $\Cu(B)\cong\Cu(\CC,B)$.


\subsection{\label{ssec:equiK}Equivariant K-theory}

In this section we give a brief account of equivariant K-theory, in a form that is suited to a bivariant extension according to the definition of the bivariant Cuntz semigroup given above. The main reference for this section is \cite[\S V.11]{blackadarK}, which in turn is largely based on \cite{phillips}. Firstly we introduce some standard terminology.

\begin{definition}[$G$-algebra] A $G$-algebra is a triple $(A,G,\alpha)$
consisting of a \Cs-algebra $A$, a topological group $G$ and a continuous action
$\alpha:G\to\Aut(A)$ of $G$ on $A$, i.e.~a point-norm continuous group
homomorphism.
\end{definition}

$G$-algebras are also known in the literature as \emph{\Cs-dynamical systems} or
\emph{\Cs-covariant systems}. Throughout this paper we let $G$ denote a
\emph{compact} topological group, unless otherwise stated. Hence, all the $G$-algebras
we consider are \Cs-algebras with a compact group action. Furthermore, we shall assume that any $G$-algebra is separable, i.e. the underlying algebra admits a dense countable $*$-subalgebra, and the group is second countable. When the action and
the group are clear from the context, or their specification is not necessary,
we denote a $G$-algebra $(A,G,\alpha)$ simply by referring to the underlying
\Cs-algebra $A$. Furthermore, we shall assume that all the $G$-algebras of this section are unital. A homomorphism, or an \emph{equivariant} $*$-homomorphism, between two $G$-algebras $(A,G,\alpha)$ and
$(B,G,\beta)$ is a $*$-homomorphism $\phi:A\to B$ that is assumed to intertwine
the actions, i.e.
	$$\phi(\alpha_g(a)) = \beta_g(\phi(a)),\qquad\forall a\in A,g\in G,$$
or, equivalently,
	$$\phi\circ\alpha_g = \beta_g\circ\phi,\qquad\forall g\in G.$$
As for the case of the $K_0$-group of a \Cs-algebra $A$, there are many ways of giving a
concrete realisation of the equivariant $K_0$-group $K_0^G(A)$ of $A$. The pictures we are
interested in are those based on idempotents and finitely generated projective
modules, together with their respective equivariant generalisations.

\begin{definition} Let $(A,G,\alpha)$ be a $G$-algebra. A finitely generated
projective $(A,G,\alpha)$-module is a pair $(E,\lambda)$ consisting of a
finitely generated projective $A$-module $E$ and a strongly continuous group
homomorphism $\lambda$ from $G$ to the invertible elements of $B(E)$ with
coefficient map $\alpha$, that is
	$$\lambda_g(ea) = \lambda_g(e)\alpha_g(a),\qquad\forall a\in A,g\in G.$$
\end{definition}

Let $(A,G,\alpha)$ be a $G$-algebra, $\pi$ be a finite-dimensional
representation of $G$ over the vector space $V$ and consider the $A$-module
$V\otimes A$. It becomes an $(A,G,\alpha)$-module when equipped with the
diagonal action $\lambda:=\pi\otimes\alpha$, which, in turn, induces an action
of $G$ on the \Cs-algebra of bounded and adjointable operators $B(V\otimes A)$
through
	$$gT := \lambda_g\circ T\circ\lambda_g^{-1},\qquad\forall g\in G.$$
Among all the elements of $B(V\otimes A)$ one can then consider the set of
$G$-invariant projections, i.e.
	$$P(V\otimes A)^G := \{p\in B(V\otimes A)\ |\ p^*\circ p = p\ \text{and}\ Gp=\{p\}\}.$$
It is easy to verify that if $p\in P(V\otimes A)^G$ is a $G$-invariant
projection, then $p(V\otimes A)$, that is the range of $p$, is a finitely
generated projective $(A,G,\alpha)$-module. The converse is also true, namely
every finitely generated projective $(A,G,\alpha)$-module is the range of a
$G$-invariant projection $p\in B(V\otimes A)$ for some representation $\pi$ of $G$ over the
finite dimensional vector space $V$ (see \cite[Proposition 11.2.3]{blackadarK}).
Hence, $G$-invariant projections and finitely generated projective
$(A,G,\alpha)$-modules are interchangeable objects.

If $\pi$ and $\omega$ are two finite dimensional representations of $G$ over the
vector spaces $V$ and $W$ respectively, one can equip $B(V\otimes A, W\otimes
A)\cong B(V,W)\otimes A$ with the $G$-action given by
	$$gT := (\omega_g\otimes\alpha_g)\circ T\circ (\pi_g\otimes\alpha_g)^{-1}.$$
Here, as in \cite{gs}, we are identifying $B(V,W)\otimes A$ with the Banach space given by the bottom-left corner of the natural matrix representation of the elements of $B((V\oplus W)\otimes A)$.
Two $G$-invariant projections $p\in P(V\otimes A)^G$ and $q\in P(W\otimes A)^G$
are Murray-von Neumann equivalent (in symbols $p \simeq_G q$) if there exists
a $G$-invariant element $v\in B(V\otimes A,W\otimes A)^G$, i.e. $gv=v$ for any $g\in G$, such that $p=v^*\circ
v$ and $q = v\circ v^*$. Similarly to the non-equivariant case, Murray-von Neumann
subequivalence is expressed as follows. One says that $p$ is Murray-von Neumann
subequivalent to $q$ (in symbols $p\preceq_G q$) if there exists $v\in B(V\otimes
A, W\otimes A)^G$ such that $p = v^*\circ
v$ and $v\circ v^*\leq q$. The modules $p(V\otimes A)$ and $q(V\otimes A)$ are
then isomorphic as $(G,A,\alpha)$-modules if and only if $p$ and $q$ are
Murray-von Neumann equivalent.

\begin{definition}\label{def:equimvn} The equivariant Murray-von Neumann
semigroup $V^G(A)$ of a unital $G$-algebra $(A,G,\alpha)$ is the set of
isomorphism classes of finitely generated projective $(A,G,\alpha)$ modules
equipped with the operation $+$ derived from the direct sum of modules.
\end{definition}

Equivalently, the equivariant Murray-von Neumann semigroup $V^G(A)$ can be
defined as the set of classes of Murray-von Neumann equivalent $G$-invariant
projections over all the modules of the form $V\otimes A$, where $V$ is a finite
dimensional representation vector space for $G$. The equivariant $K_0$-group of
the $G$-algebra $A$ is obtained from $V^G(A)$ through the usual Grothendieck enveloping group construction $\Gamma$, viz.
	$$K_0^G(A):=\Gamma(V^G(A)).$$
As the aim of this paper is to provide an equivariant extension of bivariant Cuntz \emph{semigroups}, we shall not consider this construction any further. Rather, we reformulate $V^G$ in a slightly different way, which is more prone to a
bivariant generalisation, by regarding each invariant projection
$p\in P(V\otimes A)^G$ as a selfadjoint idempotent from a larger module. We let $\hat G$ denote the set of unitary equivalence classes of irreducible representations of $G$. Then, by $H_G$ we denote the Hilbert space of the direct sum over $\hat G$ of the
representation vector spaces of arbitrarily selected representative from each
$\xi\in\hat G$, viz.
	$$H_G := \bigoplus_{\xi\in\hat G}V_\xi,$$
where $V_\xi$ is the representation vector space of a unitary irreducible
representation $\pi_\xi$ in the class $\xi$. The
stabilisation of the Hilbert space $H_G$, that is $H_G^{\oplus\infty}\cong
H_G\otimes\ell^2(\IN)$, is then isomorphic to $L^2(G)\otimes\ell^2(\IN)$ by Peter-Weyl's
theorem, and therefore
	$$K(H_G\otimes\ell^2(\IN)\otimes A)\cong A\otimes K(L^2(G))\otimes K,$$
up to a reordering of the tensor factors. By equipping it with the diagonal
action $\lambda\otimes\id_K\otimes\alpha$, where $\lambda:G\to B(L^2(G))$ is the
left-regular representation of $G$, $L^2(G)\otimes\ell^2(\IN)\otimes A$ becomes an $(A,G,\alpha)$-module, and the
equivariant Murray-von Neumann semigroup of $A$ can then be identified with the
Murray-von Neumann equivalence classes of $G$-invariant projections in $A\otimes
K(L^2(G))\otimes K$. That is, using $K_G$ as a shorthand notation for
$K(L^2(G))\otimes K$, we have
	\begin{equation}\label{eq:equimvn}
		V^G(A)\cong P(A\otimes K_G)^G/\simeq_G.
	\end{equation}
\begin{example} If $G$ is the trivial group $\{e\}$ and $A$ is a $G$-algebra
then $L^2(G)\cong\CC$ and therefore $K_\CC\cong K$. Hence, $V^G(A)$ is the
ordinary Murray-von Neumann semigroup of the \Cs-algebra $A$.
\end{example}

\begin{example}\label{ex:rring} Let $G\neq\{e\}$ act trivially on the
\Cs-algebra of complex numbers $\CC$. One sees that the Grothendieck
enveloping group of $V^G(\CC)$ coincides with the representation ring
$R_\CC(G)$, or simply $R(G)$, of the group $G$, i.e.
	$$K_0^G(\CC)=\Gamma(V^G(\CC))\cong R(G),$$
where $R(G)$ is the set of formal differences of equivalence classes of finite-dimensional unitary
representations of $G$, the ring structure coming from direct sums and tensor
products.
\end{example}

About the equivariant K-theory of actions we now recall the already mentioned
fundamental theorem of Julg that connects it to the ordinary K-theory of
crossed products. This result has been generalised to the equivariant theory of
the Cuntz semigroup in \cite{gs}. We also recall that we are here under the
assumption that every group $G$ we consider is compact.

\begin{theorem}[Julg]\label{thm:julgK} Let $(A,G,\alpha)$ be a $G$-algebra.
There is a natural isomorphism between $K_0^G(A)$ and $K_0(A\rtimes G)$.
\end{theorem}

We refer the reader to \cite[Theorem 11.7.1]{blackadarK} for a proof of the
above result.


\subsection{Equivariant KK-theory}

We provide now a brief account of equivariant
KK-theory. As in \cite[\S20]{blackadarK}, we also restrict our attention to the
second countable case, to which we also add compactness of all the groups,
although some of the following results hold in greater generality.

\begin{definition} Let $(A,G,\alpha)$ be a $G$-algebra. A Hilbert
$(A,G,\alpha)$-module is a Hilbert $A$-module $E$ with an action of $G$ on $E$
which is continuous in the sense that the map $g\mapsto \norm{(gx,gx)}$, where $(\ \cdot\ ,\ \cdot\ )$ denotes the inner product on $E$, is
continuous for any $x\in E$, and compatible with the action $\alpha$ on $A$,
i.e.
	$$g(xa) = (gx)\alpha_g(a),\qquad\forall g\in G,x\in E,a\in A.$$
\end{definition}

The grading can be extended to both $G$-algebras and equivariant Hilbert modules
of the above definition. Then many of the results and operations involving
graded Hilbert modules extend to the equivariant setting, including Kasparov's
stabilisation theorem (see, for instance, \cite[Theorem 1.2.12]{jt}). The equivariant analogue
of a Kasparov triple is provided by the following definition.

\begin{definition} Let $A$ and $B$ be graded $G$-algebras. A Kasparov $A$-$B$
$G$-triple is a triple $(E,\phi,T)$, where $E$ is a countably generated Hilbert
$(B,G,\beta)$-module, $\phi:A\to B(E)$ is an equivariant graded $*$-homomorphism
and $T\in B(E)$ is a $G$-invariant operator that satisfy
\begin{enumerate}[(EKT.1)]
	\item $[\phi(a),T]\in K(E)$ for any $a\in A$;
	\item $\phi(a)(T^2-1_{B(E)})\in K(E)$ for any $a\in A$;
	\item $\phi(a)(T-T^*)\in K(E)$ for any $a\in A$.
\end{enumerate}
\end{definition}

The above definition differs slightly from the more standard one of \cite[Definition 20.2.1]{blackadarK}. Indeed, here we are making use of the assumption that the group $G$ is compact to average the operator $T$ (see Proposition 20.2.4 of \cite{blackadarK} for more details). The equivariant KK-group of the pair of \Cs-algebras $A$ and
$B$ is then defined as in the non-equivariant case by taking homotopy classes of
Kasparov $A$-$B$ $G$-triples. One then gets a bivariant functor $KK^G$ which,
likewise $KK$, is contravariant in the first argument and covariant in the
second.

With the natural identification $KK^G(\CC, B)\cong K^G_0(B)$, that occurs when $G$ acts trivially on $\CC$, one sees
immediately that the representation ring of the group $G$ is recovered as
$R(G)\cong KK^G(\CC,\CC)$. More generally, when $A$ is a $G$-algebra with the trivial action of $G$, one has the isomorphism
	$$KK^G(A,B) \cong KK(A,B\rtimes G).$$

Furthermore, there is a group homomorphism
$j_G:KK^G(A,B)\to KK(A\rtimes G,B\rtimes G)$ that is functorial in $A$ and $B$
and compatible with Kasparov product, as shown, e.g., in \cite[Theorem 20.6.2]{blackadarK}.


\subsection{Open projections}

In Section \ref{sec:equics} we define the equivariant Cuntz semigroup and we give an open projection picture. Hence, for the reader's convenience, we recall the definition of open projections and related notions.

In \cite{akemann}, Akemann has given a generalisation of the notion of open
subsets to non-commutative \Cs-algebras by naturally replacing sets with
projections.

\begin{definition}[Open projection]\label{def:op} Let $A$ be any \Cs-algebra. A
projection $p\in A\dd$ is \emph{open} if it is the strong limit of an increasing
net of positive elements $\{a_i\}_{i\in I}\subseteq A_+$.
\end{definition}

Equivalently \cite{akemann}, a projection $p\in A\dd$ is open if it belongs to the strong
closure of the hereditary subalgebra $ A_p\subseteq A$,
where
  $$A_p := p A\dd p \cap  A = pAp \cap A.$$
Observe that, for any positive element $a\in A_+$ that has $p\in A\dd$ as support projection, one has
	$$A_p = A_a,$$
where $A_a$ is the hereditary \Cs-subalgebra of $A$ generated by $a$, viz.
	$$A_a := \overline{aAa}\label{eq:her}.$$

Throughout, the set of all the open projections of $A$ in $A\dd$ will be denoted
by $P_o( A\dd)$.

A projection $p\in A\dd$ is said to be
\emph{closed} if its complement $1-p\in A\dd$ is an open projection. The
supremum of an arbitrary set $P\subset P_o(A\dd)$ of open projections in $A\dd$
is still an open projection and, likewise, the infimum of an arbitrary family of
closed projections is still a closed projection, by results in \cite{akemann}. Therefore, the closure of an open
projection $p\in A\dd$ can be defined as
  $$\overline p := \inf\{q^*q=q\in A\dd\ |\ 1-q\in P_o( A\dd)\ \wedge\ p\leq
  q\}.$$
Let $B$ be a \Cs-subalgebra of $A$. A closed projection $p\in A\dd$ is said to
be \emph{compact} in $ B$ if there exists a positive contraction $a\in B_+$
such that $pa = p$.

For a positive contraction $a$ of a \Cs-algebra $A$, its \emph{support projection} $p_a$ is the open projection in $A\dd$ given by
	$$p_a := \SOT\lim_{n\to\infty}a^{\frac1n}.$$

	\section{\label{sec:equibcs}Definitions and Main Properties}

We introduce here an equivariant extension of the bivariant Cuntz semigroup developed in \cite{bcs}. As already mentioned in the previous section, in this paper we restrict our attention to actions by second countable compact groups.

Based on Definition \ref{def:bcs} and the presentation of the equivariant Murray-von Neumann semigroup given by Equation \ref{eq:equimvn}, the equivariant extension of the bivariant Cuntz semigroup of \cite{bcs} that we seek in this paper should be based on a suitable notion of comparison between \emph{equivariant} c.p.c.~order zero maps.

A c.p.c.~order zero map $\phi$ between two $G$-algebras $(A,G,\alpha)$ and
$(B,G,\beta)$ is said to be \emph{equivariant} if it is an intertwiner for the
actions $\alpha$ and $\beta$, that is
	$$\phi\circ\alpha_g = \beta_g\circ\phi,\qquad\forall g\in G.$$
Unless otherwise stated, it will be assumed that a c.p.c.~order zero map
$\phi:A\to B$ between the $G$-algebras $A$ and $B$ is always equivariant. The Cuntz
comparison of equivariant c.p.c.~order zero maps then takes the following form.

\begin{definition} Let $A$ and $B$ be $G$-algebras, and let
$\phi,\psi:A\to B$ be c.p.c.~order zero maps. We say that $\phi$ is
equivariantly Cuntz-subequivalent to $\psi$ (in symbols $\phi\precsim_G\psi$) if
there exists a $G$-invariant sequence $\seq b\subset B^G$ such that
	$$\lim_{n\to\infty}\norm{b_n\psi(a)b_n^*-\phi(a)} = 0$$
for any $a\in A$.
\end{definition}

We observe that, in the separable case, the above definition has a standard \emph{localisation}. Two c.p.c.~order zero maps $\phi,\psi:A\to B$ are such that $\phi\precsim_G\psi$ if and only if, for every finite subset $F\Subset A$ and $\eps > 0$ there is $b\in B^G$ such that $\norm{b\psi(a)b^*-\phi(a)}<\eps$ for any $a\in F$. Reflexivity follows trivially from the fact that, by a simple averaging argument, any $G$-algebra, with $G$ compact, admits a $G$-invariant approximate unit.

Let $\sim_G$ denote the relation arising from the antisymmetrisation
of the relation $\precsim_G$ just defined, that is $\phi\sim_G\psi$ if
$\phi\precsim_G\psi$ and $\psi\precsim_G\phi$.

\begin{lemma}\label{lem:reflexivity} Let $A$ and $B$ be $G$-algebras. For any equivariant c.p.c.~order zero map $\phi:A\to B$, finite subset $F\Subset A$ and $\eps > 0$, there exists $x\in C^*(\phi(A))^G$ such that $\norm{x\phi(a)x^*-\phi(a)}<\eps$.
\end{lemma}
\begin{proof} Fix a finite subset $F\Subset A$ and $\eps > 0$. By the existence of a $G$-invariant approximate unit for $A$, there is $e\in A^G$ such that $\norm{eae^* - a} < \eps$ for any $a\in F$. Let $h_\phi$ and $\pi_\phi$ be the positive element and the $*$-homomorphism coming from Theorem \ref{thm:equiwz} applied to $\phi$. As $h_\phi^{\frac2n}h_\phi$ converges to $h_\phi$ in norm, there is $m\in\IN$ such that $\norm{h_\phi^{\frac1m}\phi(eae^*)h_\phi^{\frac1m}-\phi(eae^*)}<\eps$ for every $a\in F$. With $x:=h_\phi^{\frac1m}\pi_\phi(e)=\phi^{\frac1m}(e)\in B^G$, one has the estimate
	\begin{align*}
		\norm{x\phi(a)x^*-\phi(a)} &= \norm{h_\phi^{\frac1m}\phi(eae^*)h_\phi^{\frac1m} - \phi(a)}\\
			&\leq \norm{h_\phi^{\frac1m}\phi(eae^*)h_\phi^{\frac1m} - \phi(eae^*)} +
				\norm{\phi(eae^*) - \phi(a)}\\
			&< 2\eps
	\end{align*}
for any $a\in F$.
\end{proof}

If $(A,G,\alpha)$ is a
$G$-algebra, we shall always assume that the tensor product $A\otimes K_G$ is
equipped with the diagonal action $\alpha\otimes(\lambda_G\otimes \id_K)$, where
$\lambda_G$ is the left-regular representation of $G$ on $L^2(G)$. As an equivariant generalisation of the bivariant Cuntz
semigroup $\Cu$ of \cite{bcs} we then give the following
definition.

\begin{definition}\label{def:equi_bcu} Let $A$ and $B$ be
$G$-algebras. The equivariant bivariant Cuntz semigroup $\Cu^G(A,B)$ of $A$ and
$B$ is the set of equivalence classes
	$$\Cu^G(A,B) := \{\phi:A\otimes K_G\to B\otimes K_G\ |\ \phi\text{ equiv.
	c.p.c.~order zero map}\}/\sim_G.$$
\end{definition}

As in the case of the ordinary Cuntz semigroup and of its bivariant extension proposed in \cite{bcs}, the above semigroups can be equipped with a positive order structure by requiring
	$$[\phi]\leq_G[\psi]$$
whenever $\phi\precsim_G\psi$. Hence, $(\Cu^G(A,B),\leq_G)$ becomes a positively ordered Abelian monoid for any pair of $G$-algebras $A$ and $B$.

As discussed in \cite{wz}, given any two \Cs-algebras $A$ and $B$, there is a one-to-one correspondence between c.p.c.~order zero maps from $A$ to $B$ and $*$-homomorphisms from the cone over $A$, i.e. $C_0((0,1])\otimes A$, to $B$. This result generalises to the equivariant setting by equipping the cone over $A$ with the diagonal action, as shown by \cite[Corollary 2.10]{grok}.

We now give an equivariant extension of the structure theorem for c.p.c.~order zero maps of \cite{wz}.

\begin{theorem}\label{thm:equiwz} Let $(A,G,\alpha)$ and $(B,G,\beta)$ be
$G$-algebras and let $\phi:A\to B$ be an equivariant c.p.c.~order zero map. Set
$C_\phi := C^*(\phi(A))$ and introduce an action of $G$ on $\M(C_\phi)$ by
restricting the bidual maps $\beta_g\dd$, $g\in G$, onto it. Then there exists a
$G$-invariant positive element $h_\phi\in\M(C_\phi)_+^G\cap C_\phi'$, with
$\norm{h_\phi} = \norm\phi$, and a non-degenerate $*$-homomorphism
$\pi_\phi:A\to\M(C_\phi)\cap\{h_\phi\}'$ such that $\beta_g\dd\circ\pi_\phi =
\pi_\phi\circ\alpha_g$ for any $g\in G$, and
	$$\phi(a) = h_\phi\pi_\phi(a),\qquad\forall a\in A.$$
\end{theorem}
\begin{proof} By the structure theorem of \cite{wz}, there are a positive element
$h_\phi\in\M(C_\phi)_+\cap C_\phi'$ with $\norm{h_\phi} = \norm\phi$ and a
non-degenerate $*$-homomorphism $\pi_\phi:A\to\M(C_\phi)\cap\{h_\phi\}'$ such
that $\phi(a) = h_\phi\pi_\phi(a)$ for any $a\in A$. If $\seq e\subset A$ is an
approximate unit, the equivariance of $\phi$ implies
$h_\phi\pi_\phi(\alpha_g(e_n))=\beta_g(h_\phi\pi_\phi(e_n))$ for any $n\in\IN$,
whence
	\begin{align*}
		0	&= \SOT\lim_{n\to\infty}[\phi(\alpha_g(e_n))-\beta_g(\phi(e_n))]\\
			&= h_\phi-\beta_g\dd(h_\phi),\qquad\forall g\in G,
	\end{align*}
	which shows that $h_\phi$ is $G$-invariant in $\M(C_\phi)$, with the action
	given by the restriction of $\beta\dd$ to this multiplier algebra. Since
	$h_\phi^{\frac1n}$ is also $G$-invariant, equivariance also implies
	$h_\phi^{\frac1n}[\pi_\phi(\alpha_g(a))-\beta_g\dd(\pi_\phi(a))]=0$ for any
	$n\in\IN$ and $a\in A$, whence
	\begin{align*}
		0 &= \SOT\lim_{n\to\infty} h_\phi^{\frac1n} [\pi_\phi(\alpha_g(a)) -
		\beta_g\dd(\pi_\phi(a))]\\
			&= \pi_\phi(\alpha_g(a))-\beta_g\dd(\pi_\phi(a)),\qquad\forall a\in A
	\end{align*}
	i.e $\pi_\phi\circ\alpha_g = \beta_g\dd\circ\pi_\phi$, for any $g\in G$.
\end{proof}

The proof of the above results nowhere uses that the group $G$ is compact and therefore it applies to the non-compact case as well. The fact that such a result holds for the equivariant case allows us to give equivariant generalisations of some of the results in \cite{bcs}.

\begin{proposition}\label{prop:compositions} Let $A$, $B$ and $C$ be $G$-algebras, and let $\phi,\psi:A\to B$, $\eta,\theta:B\to C$ be equivariant c.p.c.
order zero maps such that $\phi\precsim_G\psi$ and $\eta\precsim_G\theta$. Then
$\eta\circ\phi\precsim_G\theta\circ\phi$ and $\eta\circ\phi\precsim_G\eta\circ\psi$.
\end{proposition}
\begin{proof} The implication
$\eta\precsim_G\theta\Rightarrow\eta\circ\phi\precsim_G\theta\circ\phi$ is trivial.
For the other implication, let $h_\eta$ and $\pi_\eta$ be the positive element and support $*$-homomorphism coming from Theorem \ref{thm:equiwz} applied to $\eta$. For a finite subset $F\Subset A$ and $\eps > 0$, find $b\in B^G$ such that
	$$\norm{b\psi(a)b^* - \phi(a)} < \eps$$
for any $a\in F$. Since $h_\eta^{\frac2n}h_\eta$ converges to $h_\eta$ in norm, there exists $n \in \IN$ such that
	$$\norm{h_\eta^{\frac1n}\eta(b\psi(a)b^*)h_\eta^{\frac1n} - \eta(b\psi(a)b^*)} < \eps,$$
for any $a \in F$. Therefore, with the element $d:=h_\eta^{\frac1n}\pi_\eta(b)=\eta^{\frac1n}(b)\in C^G$, one has the estimate
	\begin{align*}
		\norm{d(\eta\circ\psi)(a)d^*-(\eta\circ\phi)(a)}
			&\leq\norm{h_\eta^{\frac1n}\eta(b\psi(a)b^*)h_\eta^{\frac1n}-\eta(b\psi(a)b^*)}\\
			&\qquad+\norm{\eta(b\psi(a)b^*) - (\eta\circ\phi)(a)}\\
			&< \eps + \norm\eta\norm{b\psi(a)b^*-\phi(a)}\\
			&< 2\eps.
	\end{align*}
	Hence $\eta\circ\phi\precsim_G\eta\circ\psi$.
\end{proof}

Observe that the $C_0((0,1])_+$-functional calculus of \cite{wz} extends to the equivariant case. Indeed, $f(\phi) := f(h_\phi)\pi_\phi$ is an equivariant c.p.~map for every equivariant c.p.c.~order zero map $\phi$ and $f\in C_0((0,1])_+$.

\begin{proposition}\label{prop:commcomp} Let $A$ and $B$ be \Cs-algebras, and let $\phi:A\to B$ be an equivariant c.p.c.\ order zero map. For any pair of positive continuous functions $f,g\in C_0((0,1])_+$ such that $\supp f\subseteq\supp g$ we have that $f(\phi)\precsim_G g(\phi)$.
\end{proposition}
\begin{proof} Fix a finite subset $F$ of $A$ and an $\eps > 0$. For a given pair of positive continuous functions $f,g\in C_0((0,1])_+$ such that $\supp f\subseteq\supp g$, find $k\in C_0((0,1])_+$ with the property that $\norm{gk-f}<\frac\eps M$, where $M:=\max_{a\in F}\norm a$, e.g.~like in the proof of \cite[Proposition 2.5]{apt2009}. By an argument similar to that of the proof of Lemma \ref{lem:reflexivity} we can find $e\in B^G$ such that
	$$\norm{eg(\phi)(a)e^* - (gk)(\phi)(a)} < \eps$$
for any $a\in F$. We then have the estimate
	\begin{align*}
		\norm{eg(\phi)(a)e^* - f(\phi)(a)} &\leq
				\norm{eg(\phi)(a)e^* - (gk)(\phi)(a)}
				+ \norm{(gk)(\phi)(a) - f(\phi)(a)}\\
			&<\eps + \frac{\eps\norm a}M\\
			&\leq 2\eps,
	\end{align*}
	for any $a\in F$.
\end{proof}


\subsection{Stability} We recall that for every $*$-isomorphism $\gamma:K\otimes K\to K$ and minimal projection $e\in K$ there exists an isometry $v\in B(\ell^2(\IN))$ such that $\Ad_v\circ\gamma\circ(\id_K\otimes e) = \id_K$. By Fell's absorption
principle this fact generalises to the equivariant setting. First of all we observe that, inside the algebra $K_G$ there is a
minimal $G$-invariant projection $e_G$. This is given by $e\otimes e_0$, where
$e$ is any minimal projection of $K$ and $e_0$ is the minimal projection in
$K(L^2(G))$ that projects onto the representation space of the trivial
representation of $G$. Furthermore, the flip $a\otimes
e_G\mapsto e_G\otimes a$ is implemented by a $G$-invariant unitary.

\begin{proposition}\label{prop:gammaG} The $G$-algebras $K_G\otimes K_G$ and $K_G$ are equivariantly isomorphic.~Furthermore, for every equivariant isomorphism
$\gamma_G:K_G\otimes K_G \to K_G$ there exists a $G$-invariant isometry $v\in B(L^2(G)\otimes\ell^2(\IN))^G$ with the property that $\Ad_v\circ\gamma_G\circ(\id_{K_G}\otimes e_G) = \id_{K_G}$.
\end{proposition}
\begin{proof} By Fell's absorption principle, the $G$-algebras
	$$(K_G\otimes K_G, (\id_K\otimes\lambda_G)\otimes(\id_K\otimes\lambda_G))\quad\text{and}\quad (K_G\otimes K_G, (\id_K\otimes\lambda_G)\otimes\id_{K_G})$$ are equivariantly isomorphic by a map $\phi$
that is such that $\phi\circ(\id_{K_G}\otimes e_G) = \id_{K_G}\otimes e_G$.
Since for every fixed isomorphism $\gamma:(K,\id_K)\otimes(K_G,\id_{K_G})\to (K,\id_K)$ there exists an isometry $w\in B(\ell^2(\IN))$ with the
property that $\Ad_w\circ\gamma\circ(\id_K\otimes e_G) = \id_K$, the sought map is then given by
$\gamma_G:=\id_{K(L^2(G))}\otimes\gamma$, with the $G$-invariant isometry $v$ given
by $1_{B(L^2(G))}\otimes w$.
\end{proof}

\begin{lemma}\label{lem:tensorcse} Let $A$, $B$, $C$ and $D$ be $G$-algebras, and let $\phi,\psi:A\to B$, $\eta:C\to D$ be equivariant c.p.c.~order zero maps
such that $\phi\precsim_G\psi$. Then $\eta\otimes\phi\precsim_G\eta\otimes\psi$ for any tensor norm.
\end{lemma}
\begin{proof} If $\seq b\subset B^G$ is the sequence that witnesses $\phi\precsim_G\psi$, then $\{d_n\otimes b_n\}_{n\in\IN}$,
where $\seq d\subset D^G$ is a $G$-invariant approximate unit for $D$, witnesses the sought equivariant Cuntz
subequivalence between $\eta\otimes\phi$ and $\eta\otimes\psi$.
\end{proof}

\begin{corollary}\label{cor:stabilitycse} Let $A$ and $B$ be $G$-algebras
and let $\phi,\psi:A\to B$ be equivariant c.p.c.~order zero maps. Then $\phi\precsim_G\psi$ in
$B$ if and only if $\phi\otimes\id_{K_G}\precsim_G\psi\otimes\id_{K_G}$ in
$B\otimes K_G$.
\end{corollary}
\begin{proof} The implication
$\phi\precsim_G\psi\Rightarrow\phi\otimes\id_{K_G}\precsim_G\psi\otimes\id_{K_G}$
follows from the previous lemma. For the other implication observe that $B$
embeds into $B\otimes K_G$ by means of the injective map $b\stackrel\iota\mapsto
b\otimes e_G$. If $\seq b\subset
(B\otimes K_G)^G$ is the sequence that witnesses the relation
$\phi\otimes\id_{K_G}\precsim_G\psi\otimes\id_{K_G}$ then, with $x_n:=(1_{\tilde B}\otimes
e_G)b_n(1_{\tilde B}\otimes e_G)\in B\otimes \{e_G\}$, where $1_{\tilde B}$ is either the unit
of $B$ or that of its minimal unitisation $\tilde B$, we have
	$$\norm{x_n^*(\psi(a)\otimes e_G)x_n-\phi(a)\otimes e_G}\to 0,\qquad\forall a\in\
	A$$
which can be pulled back to $B$ through $\iota$ to give
	$$\norm{\iota^{-1}(x_n)^*\psi(a)\iota^{-1}(x_n)-\phi(a)}\to 0,\qquad\forall
	a\in A.$$
Since $x_n\in(B\otimes\{e_G\})^G=B^G\otimes\{e_G\}$, it follows that $\iota^{-1}(x_n)\in B^G$, for any $n\in\IN$, whence $\phi\precsim_G\psi$.
\end{proof}

Thanks to the above proposition and the map $\gamma_G$ of Proposition
\ref{prop:gammaG}, the stability of $\Cu^G$ holds in the rather general form of
the following result.

\begin{theorem}\label{thm:equistability} For any pair of $G$-algebras $A$ and $B$, the partially ordered Abelian monoids $\Cu^G(A,B)$ and $\Cu^G(A\otimes K_G,B\otimes K_G)$ are order isomorphic.
\end{theorem}
\begin{proof} Mutual inverses are given by the pair of semigroup homomorphisms
	$$[\phi]\mapsto[\phi\otimes\id_{K_G}],\qquad[\phi]\in \Cu^G(A,B)$$
and
	$$[\Phi]\mapsto[(\id_B\otimes\gamma_G)\circ\Phi\circ(\id_{A\otimes K_G}\otimes e_G)],$$
where $e_G$ is the minimal $G$-invariant projection in $K_G$ mentioned at the beginning of this section. Indeed, by making use of
Proposition \ref{prop:compositions} and Lemma \ref{lem:tensorcse} one has
	\begin{align*}
		(\id_B\otimes\gamma_G)\circ(\phi\otimes\id_{K_G})\circ(\id_{A\otimes K_G}\otimes e_G) &=
		(\id_B\otimes\gamma_G)\circ(\phi\otimes e_G)\\
			&= (\id_B\otimes\gamma_G)\circ(\id_B\otimes\id_{K_G}\otimes
			e_G)\circ\phi\\
			&\sim_G (\id_B\otimes\id_{K_G})\circ\phi\\
			&= \phi
	\end{align*}
	and
	\begin{align*}
		((\id_B\otimes\gamma_G)\circ\Phi\circ(\id_{A\otimes K_G}&\otimes
		e_G))\otimes\id_{K_G}=\\
			&=
			(\id_B\otimes\gamma_G\otimes\id_{K_G}) \circ
			(\Phi\otimes\id_{K_G})\circ(\id_{A\otimes K_G}\otimes
			e_G\otimes\id_{K_G})\\
			&\sim_G (\id_B\otimes\gamma_G\otimes\id_{K_G}) \circ
			(\Phi\otimes\id_{K_G})\circ(\id_{A\otimes K_G}\otimes\id_{K_G}\otimes e_G)\\
			&= (\id_B\otimes\gamma_G\otimes\id_{K_G})\circ(\Phi\otimes e_G)\\
			&= (\id_B\otimes\gamma_G\otimes\id_{K_G}) \circ
			(\id_B\otimes\id_{K_G}\otimes\id_{K_G}\otimes e_G)\circ\Phi\\
			&\sim_G (\id_B\otimes\gamma_G\otimes\id_{K_G}) \circ
			(\id_B\otimes\id_{K_G}\otimes e_G\otimes\id_{K_G})\circ\Phi\\
			&\sim_G (\id_B\otimes\id_{K_G}\otimes\id_{K_G})\circ\Phi\\
			&= \Phi,
	\end{align*}
	which become equalities at the level of equivariant Cuntz classes.
\end{proof}

\begin{corollary}\label{cor:special} Let $A$ and $B$ be $G$-algebras. For every $\Phi\in\Cu^G(A,B)$ there exists an equivariant c.p.c.~order zero map $\phi:A\to B\otimes K_G$ such that $\Phi = [(\id_B\otimes\gamma_G)\circ(\phi\otimes\id_{K_G})]$.
\end{corollary}
\begin{proof} Define a semigroup $\cat{cu}^G(A,B)$ of Cuntz-equivalence classes of equivariant c.p.c.~order zero maps from $A$ to $B\otimes K_G$, equipped with the same binary operation of $\Cu^G(A,B)$ and repeat the argument of the previous proof to show that they are isomorphic. This amounts to replacing $\id_{A\otimes K_G}$ with $\id_A$.
\end{proof}

The following example shows that Definition \ref{def:equi_bcu} gives an
equivariant extension of the bivariant Cuntz semigroup of Definition \ref{def:bcs}.

\begin{example} Let $G$ be the trivial group $\{e\}$. Then $K_G = \CC\otimes K \cong K$ with the trivial action and therefore $\Cu^G(A,B)\cong\Cu(A,B)$.
\end{example}

The example that follows shows that Definition \ref{def:equi_bcu} can be
regarded as an bivariant extension of the equivariant Cuntz semigroup defined in \cite{gs}.

\begin{example}\label{ex:equics} Let $(A,G,\alpha)$ and $(B,G,\beta)$ be $G$-algebras. Theorem
\ref{thm:equistability} implies that, for every class $\Phi\in\Cu^G(A,B)$, there
exists a representative of the form\footnote{Here we are tacitly dropping the map $\id_B\otimes\gamma_G$ of Corollary \ref{cor:special} in order to ease the notation.} $\phi\otimes\id_{K_G}$, where $\phi:A\to
B\otimes K_G$ is an equivariant c.p.c.~order zero map. When $A=\CC$ with the
trivial action of $G$ then
	$$\phi(z) = zh_\phi,\qquad \forall z\in \CC,$$
where $h_\phi$ is a $G$-invariant positive element in $B\otimes K_G$ by Theorem \ref{thm:equiwz}. Hence, $\Cu^G(\CC,B)$ can be identified with
Cuntz-equivalence classes of $G$-invariant positive elements from $B\otimes
K_G$, i.e.
	$$\Cu^G(\CC, B)\cong\Cu^G(B).$$

More generally, if $G$ acts trivially on $A$, then any equivariant c.p.c.~order zero map $\phi:A\to B\otimes K_G$ maps into the fixed point algebra $(A\otimes K_G)^G\cong (A\rtimes G)\otimes K$, so that one sees that there is a natural isomorphism
	$$\Cu^G(A,B)\cong \Cu(A,B\rtimes G).$$
\end{example}


\subsection{Functoriality\label{sec:functoriality}}

We now investigate the functoriality of the equivariant bivariant Cuntz semigroup $\Cu^G(\ \cdot\ ,\ \cdot\ )$. The following two theorems show that $\Cu^G(\ \cdot\ ,\ \cdot\ )$ defines a bifunctor from the category of $G$-algebras to that of positively ordered Abelian monoids, contravariant in the first and covariant in the second.

\begin{theorem} Let $B$ be a $G$-algebra. $\Cu^G(\ \cdot\ ,B)$ is a
contravariant functor from the category of $G$-algebras to that of ordered
Abelian monoids.
\end{theorem}
\begin{proof} Let $A_1,A_2$ be any $G$-algebras, and consider a
homomorphism $f:A_1\to A_2$ and an equivariant c.p.c.~order zero map
$\psi:A_2\otimes K_G\to B\otimes K_G$. By defining $f^*(\psi)$ as
	$$f^*(\psi) := \psi\circ (f\otimes\id_{K_G}),$$
the following diagram
	$$\xymatrix@R=0.5em@C=3em{%
		A_1\otimes K_G  \ar[dr]^-{f^*(\psi)}\ar[dd]_{f} &             \\
		                                    & B\otimes K_G \\
		A_2\otimes K_G \ar[ur]_-{\psi}                 & 
	}$$
commutes. Hence $f^*(\psi)$ is an equivariant c.p.c.~order zero map and therefore $f^*$ defines a pull-back which can be projected onto equivalence classes by setting
	$$\Cu^G(f,B)([\psi]) = [f^*(\psi)],\qquad\forall [\psi]\in \Cu^G(A_2,\A
	B).$$
It is easy to check that this yields a well-defined map.
\end{proof}

\begin{theorem} Let $A$ be a $G$-algebra. $\Cu^G(A,\ \cdot\ )$ is a covariant
functor from the category of $G$-algebras to that of ordered Abelian
monoids.
\end{theorem}
\begin{proof} Let $B_1$ and $B_2$ be any $G$-algebras, $g:B_1\to B_2$ an equivariant $*$-homomorphism and $\psi:A\otimes K_G\to B_1\otimes K_G$ a c.p.c.~order zero map. Define $g_*(\psi)$ as
	$$g_*(\psi) := (g\otimes\id_{K_G})\circ\psi.$$
so that the diagram
	$$\xymatrix@R=0.5em@C=3em{%
		     & B_1\otimes K_G\ar[dd]^{g} \\
		A\otimes K_G \ar[ru]^-{\psi}\ar[rd]_-{g_*(\psi)} & \\
		     & B_2\otimes K_G
	}$$
commutes. Such map is clearly equivariant c.p.c.~order zero and defines a push-forward between c.p.c.~order zero
maps that gives rise to the well-defined semigroup homomorphism
	$$\Cu^G(A, g)([\psi]) = [g_*(\psi)],\qquad\forall[\psi]\in \Cu^G(A,B),$$
since $g(A^G)\subset B^G$.
\end{proof}


\begin{remark} The last two theorems still hold if one considers equivariant c.p.c.~order zero maps instead of equivariant $*$-homomorphisms as arrows between $G$-algebras.
\end{remark}


\subsection{Additivity}

Let $A_1$, $A_2$ and $B$ be $G$-algebras. We shall say that two equivariant c.p.c.\ order zero maps $\phi : A_1 \to B$ and $\psi : A_2 \to B$ are orthogonal, and we shall indicate this by $\phi\perp\psi$, if $\phi(A_1)\psi(A_2) = \{0\}$. This implies, in particular, that the positive elements $h_\phi,h_\psi\in B\dd$ coming from Theorem \ref{thm:equiwz} applied to $\phi$ and $\psi$ respectively are orthogonal, i.e. $h_\phi h_\psi = 0$ in $B\dd$. 

\begin{proposition} Let $A_1$, $A_2$ and $B$ be $G$-algebras and let $\phi_1:A_1\to B$ and $\phi_2:A_2 \to B$ be c.p.c.~order zero maps such that $\phi_1\perp\phi_2$. Then $\phi_1(A_1)\cap\phi_2(A_2) = \{0\}$.
\end{proposition}
\begin{proof} Assume that there are $a_1\in A_1, a_2\in A_2$ such that $b=\phi_1(a_1) = \phi_2(a_2)$. Then
	\begin{align*}
		\norm b^2 &= \norm{b^*b}\\
			&= \norm{\phi_1(a_1)^*\phi_2(a_2)}\\
			&= \norm{\phi_1(a_1^*)\phi_2(a_2)}\\
			&= 0
	\end{align*}
	by orthogonality of $\phi_1$ and $\phi_2$. Hence, $b=0$.
\end{proof}

Let $A_1$ and $A_2$ be $G$-algebras. We observe that, given two equivariant c.p.c.\ order zero maps $\phi_1:A_1\to B$ and $\phi_2:A_2\to B$, their direct sum
$\phi_1\oplus\phi_2$ is easily seen to be an equivariant c.p.c.\ order zero map from $A_1\oplus A_2$ to $M_2(B)$, where the action on $A_1\oplus A_2$ is $\alpha_1\oplus\alpha_2$ and that on $M_2(B)\cong M_2\otimes B$ is $\id_2\otimes\beta$. For the
converse of this property we provide the following results.

\begin{lemma} Let $A_1,A_2,B$ be $G$-algebras. A map
$\phi:A_1\oplus A_2\to B$ is an equivariant c.p.c.\ order zero if and only if there are equivariant c.p.c.\ order zero maps $\phi_1:A_1\to B$ and $\phi_2:A_2\to B$ such that
	\begin{enumerate}[i.]
		\item $\phi_1(a_1) + \phi_2(a_2) = \phi(a_1, a_2)$, for any $a_1\in A_1$ and $a_2\in A_2$;
		\item $\phi_1 \perp \phi_2$.
	\end{enumerate}
\end{lemma}
\begin{proof} Let $\phi:A_1\oplus A_2\to B$ be a c.p.c.~order zero map. Define the equivariant c.p.c.~order zero maps $\phi_1:A_1\to B$ and $\phi_2:A_2\to B$ as
		$$\phi_1(a_1) := \phi(a_1, 0)\qquad\text{and}\qquad\phi_2(a_2):=\phi(0, a_2)$$
	respectively, with $a_1\in A_1,a_2\in A_2$. Then clearly one has that $\phi(a_1, a_2) =
	\phi_1(a_1)+\phi_2(a_2)$. Furthermore, the order zero property implies that
		$$\phi_1(a_1)\phi_2(a_2) = \phi(a_1, 0)\phi(0, a_2) = 0,$$
	for any $a_1\in A_1,a_2\in A_2$, whence $\phi_1\perp\phi_2$.
	
	Conversely, assume that there are equivariant c.p.c.\ order zero maps $\phi_1$ and $\phi_2$ with the desired properties. Their sum $\phi_1 + \phi_2$ is clearly an equivariant c.p.~map. Contractivity follows from the orthogonality $\phi_1\perp\phi_2$. Let $h_1$ and $h_2$ be the positive elements coming from Theorem \ref{thm:equiwz} applied to $\phi_1$ and $\phi_2$ respectively. Then
		$$\norm\phi = \norm{h_1 + h_2}\leq 1,$$
	since $h_1\perp h_2$. The order zero property follows immediately from the fact that both $\phi_1$ and $\phi_2$ are assumed to be orthogonality preserving.
\end{proof}

\begin{lemma}\label{lem:cpcds} Let $A_1$, $A_2$ and $B$ be $G$-algebras, and let $\phi_1:A_1\to B$ and $\phi_2:A_2\to B$ be equivariant c.p.c.\ order zero maps such that $\phi_1\perp\phi_2$. Then
	$$\begin{bmatrix}\phi_1 + \phi_2 & 0\\0 & 0\end{bmatrix}\sim_G\begin{bmatrix}\phi_1 & 0\\0&\phi_2\end{bmatrix}$$
in $M_2(B)$, where $\phi_1+\phi_2$ is the equivariant c.p.c.\ order zero map from $A_1\oplus A_2$ to $B$ given by $(\phi_1+\phi_2)(a_1, a_2) := \phi_1(a_1)+\phi_2(a_2)$.
\end{lemma}
\begin{proof} By the previous lemma, the matrix on the left is a well-defined equivariant c.p.c.\ order zero map from $A_1\oplus A_2$ to $M_2(B)$. Fix a $G$-invariant approximate unit $\seq e$ for $A$, and introduce the $G$-invariant sequences of $M_2(B)$
	$$x_n:=\begin{bmatrix}
		\phi_1^{\frac12}(e_n) & 0\\
		\phi_2^{\frac12}(e_n) & 0
	\end{bmatrix},\quad
	y:=\begin{bmatrix}
		\phi_1^{\frac14}(e_n) & \phi_2^{\frac14}(e_n)\\
		0 & 0
	\end{bmatrix}.$$
	One easily sees from Theorem \ref{thm:equiwz} and the definition of functional calculus for equivariant c.p.c.~order zero maps that
		$$\lim_{n\to\infty}\norm{x_n\begin{bmatrix}\phi_1 + \phi_2 & 0\\0 & 0\end{bmatrix}(a)x_n^* - \begin{bmatrix}\phi_1^2 & 0\\0&\phi_2^2\end{bmatrix}(a)} = 0$$
	and
		$$\lim_{n\to\infty}\norm{y_n\begin{bmatrix}\phi_1^{\frac12} & 0\\0&\phi_2^{\frac12}\end{bmatrix}(a)y_n^* - \begin{bmatrix}\phi_1 + \phi_2 & 0\\0 & 0\end{bmatrix}(a)} = 0$$
	for any $a\in A_1\oplus A_2$, whence
		$$\begin{bmatrix}\phi_1^2 & 0\\0&\phi_2^2\end{bmatrix}\precsim_G\begin{bmatrix}\phi_1 + \phi_2 & 0\\0 & 0\end{bmatrix}\qquad\text{and}\qquad\begin{bmatrix}\phi_1 + \phi_2 & 0\\0 & 0\end{bmatrix}\precsim_G\begin{bmatrix}\phi_1^{\frac12} & 0\\0&\phi_2^{\frac12}\end{bmatrix}.$$
	 By Proposition \ref{prop:commcomp} we then have
		$$\begin{bmatrix}\phi_1^2 & 0\\0&\phi_2^2\end{bmatrix}\sim_G\begin{bmatrix}\phi_1^{\frac12} & 0\\0&\phi_2^{\frac12}\end{bmatrix}\sim_G\begin{bmatrix}\phi_1 & 0\\0&\phi_2\end{bmatrix},$$
	which concludes the proof.
\end{proof}

\begin{proposition}\label{prop:w1add} For any triple of $G$-algebras $A_1$,
$A_2$ and $B$, the partially ordered semigroup isomorphism
		$$\Cu^G(A_1\oplus A_2, B)\cong \Cu^G(A_1,B)\oplus \Cu^G(A_2, B)$$
	holds.
\end{proposition}
\begin{proof} Let $\sigma : \Cu^G(A_1,B)\oplus \Cu^G(A_2,B)\to \Cu^G(A_1\oplus A_2,B)$ be
the map given by
		$$\sigma([\phi_1]\oplus[\phi_2]) = [\phi_1\oplus\phi_2].$$
	By the above two lemmas it is clear that this map is surjective. To
	prove injectivity we show that $\phi_1\oplus\phi_2\precsim_G\psi_1\oplus\psi_2$
	implies $\phi_k\precsim_G\psi_k$, $k=1,2$. By hypothesis there exists a sequence
	$\seq b\subset (B\otimes K_G)^G$ such that
		$$b_n^*(\psi_1(a_1)\oplus\psi_2(a_2))b_n\to\phi_1(a_1)\oplus\phi_2(a_2)$$
	in norm for every $a_1\in A_1,a_2\in A_2$. As $M_2(B\otimes K_G)\cong
	B\otimes K_G$ equivariantly, the sequence $b_n$ has the structure
		$$b_n = \sum_{i,j=1}^2b_{n,ij}\otimes e_{ij},$$
	where $b_{n,ij}\in (B\otimes K_G)^G$ for any $i,j=1,2$ and $\{e_{ij}\}_{i,j=1,2}$
	form the standard basis of matrix units for $M_2$. Thus, for $a_2 = 0$, one
	finds that
		$$\lim_{n\to\infty}\norm{b_{n,11}^*\psi_1(a_1)b_{n,11} - \phi_1(a_1)} = 0$$
	for any $a_1\in A_1$, i.e. $\phi_1\precsim_G\psi_1$. A similar argument with
	$a_1=0$ leads to the conclusion that $\phi_2\precsim_G\psi_2$ as well.
	To check that $\sigma$ preserves the semigroup operations it suffices to show
	that
		$$(\phi_1\hoplus\phi_2) \oplus (\psi_1\hoplus\psi_2) \sim_G
		(\phi_1\oplus\psi_1) \hoplus (\phi_2\oplus\psi_2).$$
	A direct computation reveals that such equivalence is witnessed by the
	sequence $\seq b\subset M_4((B\otimes K_G)^G)$ given by
		$$b_n:= u_n\otimes(e_{11} + e_{44} + e_{23} + e_{32}),$$
	where $\seq u\subset (B\otimes K_G)^G$ is an approximate unit for $B\otimes K_G$.
\end{proof}

As in the case of the bivariant Cuntz semigroup and of KK-theory, $\Cu^G(\ \cdot\ ,\ \cdot\ )$ is countably additive in the first argument. We now proceed to prove this result which, \emph{en passant}, provides an alternative proof for the finite additivity in the first argument.

\begin{lemma} Let $A$ and $B$ be $G$-algebras, $\phi:A\to B$ a countable sum of pair-wise orthogonal equivariant c.p.c.\ order zero maps, that is
	$$\lim_{n\to\infty}\norm{\phi(a) - \sum_{k=1}^n\phi_k(a)} = 0,\qquad\forall a\in A,$$
 where each $\phi_k$ is an equivariant c.p.c.\ order zero map and $\phi_k\perp\phi_i$ for any $i\neq k$ in $\IN$, and $\seq \lambda\subset\IR^+$ a sequence that sums up to 1. Then
	$$\phi \sim_G \sum_{k=1}^\infty \lambda_k\phi_k.$$
\end{lemma}
\begin{proof} Fix a finite subset $F$ of $A$ and $\eps > 0$. Find an $n\in\IN$ such that
	$$\norm{\sum_{k=n+1}^\infty\phi_k(a)}<\eps$$
for any $a\in F$. Define the \Cs-subalgebras $B_k:= C^*(\phi_k(A))\subset B$ for any $k\in\IN$. By Lemma \ref{lem:reflexivity}, find $e_k\in B_k^G$ such that
	$$\norm{e_k\phi_k(a)e_k^*-\phi_k(a)}<\frac\eps n$$
for any $a\in F$ and $k=1,\ldots,n$. Observe that the orthogonality of the maps $\phi_k$ implies that $e_i\perp e_k$ for any $i\neq k$. With the element $x\in B^G$ defined as
	$$x:=\sum_{k=1}^n \frac{e_k}{\sqrt{\lambda_k}}$$
one has the estimate
	\begin{align*}
		\norm{x\left(\sum_{k=1}^\infty \lambda_k\phi_k(a)\right)x^*-\phi(a)}
			&\leq\sum_{k=1}^n\norm{e_k\phi_k(a)e_k^*-\phi_k(a)}+\norm{\sum_{k=n+1}^\infty\phi_k(a)}\\
			&< \sum_{k=1}^n\frac\eps n + \eps\\
			&\leq 2\eps
	\end{align*}
for any $a\in F$. Hence, $\phi\precsim_G\sum_{k=1}^\infty \lambda_k\phi_k$. For the converse subequivalence, find, if necessary, a new $n$ for which
	$$\norm{\sum_{k=n+1}^\infty \lambda_k\phi_k(a)}<\eps$$
for any $a\in F$, and new elements $e_k\in B_k^G$, $k=1,\ldots,n$ such that
	$$\norm{e_k\phi_k(a)e_k^*-\phi_k(a)}<\frac\eps n.$$
With the element $y\in B^G$ defined as
	$$y:=\sum_{k=1}^n \sqrt{\lambda_k}e_k$$
one has the estimate
	\begin{align*}
		\norm{y\phi(a)y^* - \sum_{k=1}^\infty a_k\phi_k(a)}
			&\leq\sum_{k=1}^n\norm{e_k\phi_k(a)e_k^*-\phi_k(a)}+\norm{\sum_{k=n+1}^\infty a_k\phi_k(a)}\\
			&< \sum_{k=1}^n\frac\eps n + \eps\\
			&\leq 2\eps,
	\end{align*}
for any $a\in F$, which implies that $\sum_{k=1}^\infty \lambda_k\phi_k\precsim\phi$.
\end{proof}

\begin{lemma} Let $\seq A\cup\{B\}$ be a countable family of $G$-algebras. Then any equivariant c.p.c.\ order zero map $\phi:\bigoplus_{k=1}^\infty A_k\to B$ is such that
	$$\phi\otimes e \sim \bigoplus_{k=1}^\infty\frac{\phi|_{A_k}}{2^k}$$
in $B\otimes K$, where $e\in K$ is a minimal projection and $\phi|_{A_k}$ is defined as
	$$\phi|_{A_k}(a_k) := \phi(0, \cdots, 0, a_k, 0, \cdots),\qquad\forall k\in\IN,a_k\in A_k.$$
\end{lemma}
\begin{proof} Assume, without loss of generality, that $e=e_{11}$, and denote by $A$ and $\psi$ the direct sum of the $A_k$s and of the $\frac{\phi|_{A_k}}{2^k}$s respectively. Fix a $G$-invariant approximate unit $\seq e$ of $A$, set
	$$\xi_{k,n}:=\frac{\phi^{\frac12}|_{A_k}(e_n)}{2^{k/2}},\qquad\eta_{k,n} := \frac{\phi^{\frac14}|_{A_k}(e_n)}{2^{k/4}}$$
	for any $k,n\in\IN$, and define the $G$-invariant sequences $\seq x, \seq y\in (B\otimes K)^G$ by
		$$x_n:=\sum_{k=1}^\infty \xi_{k,n}\otimes e_{k1}=\begin{bmatrix}
		\xi_{1,n} & 0 & \cdots\\
		\xi_{2,n} & 0 & \cdots\\
		\vdots & \vdots & \ddots
	\end{bmatrix},\quad
	y:=\sum_{k=1}^\infty \eta_{k,n}\otimes e_{1k}=\begin{bmatrix}
		\eta_{1,n} & \eta_{2,n} & \cdots\\
		0 & 0 & \cdots\\
		\vdots & \vdots & \ddots
	\end{bmatrix},$$
	By Theorem \ref{thm:equiwz} one gets to the conclusion that
		$$\lim_{n\to\infty}\norm{x_n(\phi\otimes e)(a)x_n^* - \psi^2(a)} = 0$$
	and
		$$\lim_{n\to\infty}\norm{y\psi^{\frac12}(a)y^* - \left(\sum_{k=1}^\infty\frac{\phi|_{A_k}}{2^k}\otimes e\right)(a)} = 0$$
	for any $a\in A$. Since $\psi\sim_G \psi^2 \sim_G \psi^{\frac12}$ by Proposition \ref{prop:commcomp} and $\phi\sim_G\sum_{k=1}^\infty\frac{\phi|_{A_k}}{2^k}$ by the previous lemma, the result now follows.
\end{proof}

\begin{proposition} Let $\seq A\cup\{B\}$ be a countable family of $G$-algebras. Then the semigroups $\prod_{n\in\IN}\Cu^G(A_n,B)$ and $\Cu^G\left(\bigoplus_{n\in\IN}A_n,B\right)$ are isomorphic.
\end{proposition}
\begin{proof} Let $\sigma:\prod_{n\in\IN}\Cu^G(A_n,B)\to \Cu^G\left(\bigoplus_{n\in\IN}A_n,B\right)$ be the semigroup homomorphism defined by
	$$\sigma\left([\phi_1],[\phi_2],\ldots\right) := 
	\left[(\id_B\otimes\gamma)\circ\bigoplus_{n\in\IN}\frac1{2^n}\phi_n\right],$$
where $\gamma : K_G \otimes K\to K_G$ is any equivariant $*$-isomorphism. For any minimal projection $e\in K$, $\gamma\circ(\id_{K_G}\otimes e)$ is conjugate to $\id_{K_G}$ by a $G$-invariant isometry $w\in (B(L^2(G)), \lambda_G)$ and therefore $\gamma\circ(\id_G\otimes e)\sim_G\id_{K_G}$. An inverse of $\sigma$ is provided by the semigroup homomorphism $\rho:\Cu^G\left(\bigoplus_{n\in\IN}A_n,B\right)\to\prod_{n\in\IN}\Cu^G(A_n,B)$ given by
	$$\rho([\phi]) := ([\phi|_{A_1}],[\phi|_{A_2}],\ldots),$$
where $\phi|_{A_k}(a_k) := \phi(a_k\otimes e_{kk})$ for any $k\in\IN$ and $a_k\in A_k$. Indeed, by the previous lemma
	$$(\id_B\otimes\gamma)\circ\bigoplus_{n\in\IN}\frac{\phi|_{A_n}}{2^n}\sim_G(\id_B\otimes\gamma)\circ(\phi\otimes e)\sim_G\phi$$
and
	$$\gamma_G\circ\left(\frac{\phi_k}{2^k}\otimes e_{kk}\right)\sim_G\gamma_G\circ(\phi_k\otimes e)\sim_G \phi_k,\qquad\forall k\in\IN$$
since every minimal projection $e_{kk}$ is Cuntz-equivalent to $e$, and $\lambda\phi_k\sim_G\phi_k$ for any $\lambda\in(0,1)$.
\end{proof}

\begin{proposition}\label{prop:w2add} For any triple of $G$-algebras $A$,
$B_1$ and $B_2$, the partially ordered semigroup isomorphism
		$$\Cu^G(A, B_1\oplus B_2)\cong \Cu^G(A,B_1)\oplus \Cu^G(A, B_2)$$
	holds.
\end{proposition}
\begin{proof} Since $(B_1\oplus B_2)\otimes K_G$ is equivariantly isomorphic to
$(B_1\otimes K_G)\oplus (B_2\otimes K_G)$ one has that for every equivariant c.p.c.~order zero map
$\phi:A\to (B_1\oplus B_2)\otimes K_G$ there are equivariant c.p.c.~order zero maps $\phi_k:A\to
B_k\otimes K_G$, $k=1,2$ such that $\phi$ can be identified with\footnote{Such maps are given by $\phi_k :=
(\pi_k\otimes\id_{K_G})\circ\phi$, where $\pi_k$ is the natural projection $\pi_k:B_1\oplus B_2\to B_k$.} $\phi_1\hoplus\phi_2$. This shows that the map
$\rho: \Cu^G(A,B_1)\oplus \Cu^G(A,B_2)\to \Cu^G(A,B_1\oplus B_2)$ given by
		$$\rho([\phi_1]\oplus[\phi_2]) := [\phi_1\hoplus\phi_2]$$
	is surjective. Injectivity follows from the fact that
	$\phi_1\hoplus\phi_2\precsim_G\psi_1\hoplus\psi_2$ implies
	$\phi_1\precsim_G\psi_1$ and $\phi_2\precsim_G\psi_2$, which is obvious.
\end{proof}


\subsection{\label{ssec:cp}Relation with Crossed Products}

In KK-theory there is a group homomorphism between the equivariant KK-group and the KK-group
of the crossed product, \cite[\S2.6]{blackadarK}. We now provide an analogue of this result within the framework of the equivariant extension of the bivariant Cuntz semigroup of \cite{bcs} proposed in this paper.

\begin{proposition} Let $A$ and $B$ be $G$-algebras. Every
equivariant c.p.c.~order zero map $\phi:A\to B$ induces a c.p.c.~order zero map
$\phi_\rtimes:A\rtimes G\to B\rtimes G$ between the crossed products.
\end{proposition}
\begin{proof} Consider the map $\phi_*:L^1(G,A)\to L^1(G,B)$ given by
post-composition, that is,
		$$\phi_*(f)(g) := \phi(f(g)),\qquad\forall f\in L^1(G,A), g\in G.$$
	It is immediate to check that it gives rise to a c.p.c.~map. Furthermore, if $a,b\in L^1(G,A)$ are such that $a*b = 0$ and $\mu$ is the Haar measure on
	$G$, then one has
		\begin{align*}
			(\phi_*(a)*\phi_*(b))(g) &= \int_G
			\phi\big(a(h)\big)\beta_h\left(\phi(b(h^{-1}g))\right)\de\mu(h)\\
				&= \int_G h_\phi\phi\big(a(h)\alpha_h(b(h^{-1}g))\big)\de\mu(h)\\
				&= h_\phi\phi((a*b)(g)),\qquad\forall g\in G
		\end{align*}
	whence $\phi_*(a)*\phi_*(b) = 0$. Therefore, $\phi_*$ extends to a c.p.c.~order
	zero map $\phi_\rtimes:A\rtimes G\to B\rtimes G$.
\end{proof}

For an alternative proof of the above result, that relies on the one-to-one correspondence between c.p.c.~order zero maps and $*$-homomorphisms over the cone, we refer the reader to \cite[Proposition 2.3]{greg}.

\begin{proposition}\label{prop:rtimes} Let $A$ and $B$ be $G$-algebras and let
$\phi,\psi:A\to B$ be equivariant c.p.c.~order zero maps such that
$\phi\precsim_G\psi$. Then $\phi_\rtimes\precsim\psi_\rtimes$.
\end{proposition}
\begin{proof} Let $\seq f\subset L^1(G)$ be an approximate unit for $L^1(G)$. If $\seq
b\subset B^G$ is the sequence witnessing the subequivalence $\phi\precsim_G\psi$,
then a direct computation shows that the sequence $\seq d\subset
L^1(G,B)$ given by\footnote{For $f\in L^1(G)$ and $b\in B$, the tensor product $b\otimes f$, sometimes also denoted simply by $bf$, can be identified with the $B$-valued function of class $L^1$ on $G$ with respect to the Haar measure given by $g\mapsto f(g)b$ almost everywhere.}
		$$d_n := b_n\otimes f_n$$
	is such that
		$$\lim_{n\to\infty}\norm{d_n\psi_\rtimes(a\otimes f){d_n}^* -
		\phi_\rtimes(a\otimes f)} = 0,\qquad\forall a\otimes f\in L^1(G,A),$$
	whence $\phi_\rtimes\precsim\psi_\rtimes$.
\end{proof}

This last result shows that the assignment $\phi\mapsto\phi_\rtimes$ becomes
well-defined when considered at the level of classes. Furthermore, one can easily check that
	$$({\id_A})_{\rtimes} = \id_{A\rtimes G}$$
for any $G$-algebra $(A,G,\alpha)$. Therefore, one has the following result.

\begin{theorem}\label{th:jG} Let $A$ and $B$ be $G$-algebras. There is a
natural semigroup homomorphism
	$$j^G:\Cu^G(A,B)\to\Cu(A\rtimes G,B\rtimes G)$$
which is functorial in $A$ and $B$ and compatible with the composition product.
\end{theorem}
\begin{proof} The sought map $j^G$ is defined as
		$$j^G([\phi]) := [\phi_\rtimes]$$
	which is well-defined as a consequence of the above proposition.
\end{proof}


\subsection{\label{ssec:cuhom}Equivariant Cuntz Homology}

A notion of \emph{Cuntz homology} for compact Hausdorff spaces has been introduced in \cite{bcs}. Its definition follows the way K-homology is recovered from KK-theory, namely by fixing the second argument to be the algebra of complex numbers $\CC$. More generally, one can see that $\Cu(A,\CC)$ encodes information relative to the finite dimensional representation theory of the \Cs-algebra $A$ in the first argument. However, this topic will be touched upon elsewhere \cite{tcuhom}.

We now proceed to define an equivariant version of Cuntz homology analogously to the non-equivariant case of \cite{bcs}, and provide a concrete realisation for compact group actions on compact Hausdorff spaces.

\begin{definition} A \emph{topological dynamical system} is a triple $(X,G,\alpha)$ consisting of a topological space $X$, a topological group and a continuous $G$-action $\alpha$ of $G$ on $X$.
\end{definition}

When the specification of the action is not necessary we shall refer to the topological space $X$ to denote the topological dynamical system $(X,G,\alpha)$. A topological dynamical system $(X,G,\alpha)$ is compact if its underlying topological space $X$ and the group $G$ is.

\begin{definition} Let $(X,G,\alpha)$ be a compact topological dynamical system. The equivariant Cuntz homology of $(X,G,\alpha)$ is the Abelian monoid
	$$\Cu^G(X) := \Cu^G(C(X),\CC).$$
\end{definition}

We now recall part of the terminology and definitions introduced in \cite[\S5.3]{bcs} for the reader's convenience.

\begin{definition} The spectrum $\sigma(\phi)$ of a c.p.c.\ order zero map
$\phi:C(X)\to K$ is the closed subset $C_\phi\subset X$ associated to the kernel
of $\phi$, i.e.
	$$\sigma(\phi) := \{x\in X\ |\ f(x) = 0\ \forall f\in\ker\phi\}.$$
\end{definition}

The set of isolated points of the spectrum of a c.p.c.\ order zero map from $C(X)$ to $K$ will be denoted by $\sigma_i(\phi)$, while the essential part is defined as
$\sigma_\ess(\phi):=\sigma(\phi)\smallsetminus\sigma_i(\phi)$. Set $\tilde\IN_0 := \IN_0\cup\{\infty\}$ and observe that, given a function $\nu\in \tilde\IN_0^X$, one can operate the same decomposition on $\supp\nu$, thus obtaining a part $\nu_i$ supported by isolated points, and a part $\nu_\ess$ supported by the essential part, such that $\nu = \nu_i + \nu_\ess$ and $\nu_i$ and $\nu_\ess$ have disjoint supports.

Denote by $\Mf^G(X)$ the set of functions $\nu$ in $\tilde\IN_0^{X/G}$ with closed support\footnote{that is, those functions $f\in\tilde\IN_0^{X/G}$ for which the subset $\{x\in X/G\ |\ f(x) = 0\}$ of $X/G$ is open.} and such that the essential part $\nu_\ess$, when non-zero, has range in $\{\infty\}$. Then one has the identification
	$$\Mf^G(X) = \Mf(X/G) \cong \Cu(X/G),$$
where $X/G$ denotes the orbit space of $X$ under the action of $G$, endowed with the quotient topology. Observe that, under the above hypotheses on $X$, $X/G$ is a compact Hausdorff space.

\begin{theorem} For any compact topological dynamical system $(X,G,\alpha)$ there is a natural monoid isomorphism
	$$\Cu^G(X) \cong \Cu(X/G).$$
\end{theorem}
\begin{proof} Since $K(L^2(G))\otimes K \cong K$, every equivariant representation $\pi : C(X) \to K_G$ is easily seen to decompose, up to unitary equivalence, into one of the form
	$$\pi(f) = \sum_{k=1}^\infty M_f^{x_k}\otimes e_{kk},\qquad\forall f\in C(X),$$
where $\seq x\subset X$ and $M_f^{x_k}\in L^\infty(G)\subset B(L^2(G))$ is the multiplication operator associated to the function
	$$g \mapsto f(gx_k),\qquad\forall g\in G, k\in\IN.$$
Hence for every $x_k$, its full orbit $Gx_k$ appears in this decomposition, and the multiplicity of each of the points in $Gx_k$ is evidently constant. Therefore, the multiplicity functions can be assumed to be defined on the orbit space $X/G$, whence the presentation of Cuntz homology given in \cite{bcs} applies.
\end{proof}

\begin{corollary} Let $G$ be a compact group and let $(X,G,\alpha)$ and $(Y,G,\beta)$ be topological dynamical systems. The equivariant Cuntz homologies $\Cu^G(X)$ and $\Cu^G(Y)$ are isomorphic if and only if the orbit spaces $X/G$ and $Y/G$ are homeomorphic.
\end{corollary}

	\section{\label{sec:equics}The Equivariant Cuntz Semigroup} Analogously to the ordinary Cuntz semigroup $\Cu(A)$ of a \Cs-algebra $A$, which can be recovered from the bivariant Cuntz semigroup as $\Cu(A)\cong\Cu(\CC,A)$, in Example \ref{ex:equics} we have defined the \emph{equivariant} Cuntz semigroup of the $G$-algebra $(A,G,\alpha)$ as
	$$\Cu^G(A):=\Cu^G(\CC,A).$$
We have also shown that this object has a natural identification with the set of Cuntz-equivalence classes of $G$-invariant positive elements from $A\otimes K_G$, and therefore it turns out to coincide with the equivariant Cuntz semigroup defined in \cite{gs}. There, a thorough investigation of this new functor is carried out, and the category $\Cu$ is enriched to reflect the equivariance of the theory developed. In particular, the $G$-invariant positive element picture is complemented by an equivariant generalisation of the module picture of \cite{cei}.

In this section we recall some main definitions and results drawn from \cite{gs} and propose an open projection picture for the equivariant theory of the Cuntz semigroup. In order to do so we shall generalise many of the results of \cite{ORT} to the equivariant setting first. Then the sought open projection picture will follow naturally. For concreteness, we now give the explicit definition of the equivariant Cuntz semigroup that we will employ throughout this section. We also recall that we are still under the assumption that every group with consider is compact and second countable.

\begin{definition}[Equivariant Cuntz Semigroup]\label{def:equiCu} Let
$(A,G,\alpha)$ be a $G$-algebra. Its equivariant Cuntz semigroup is the set of
classes
	$$\Cu^G(A) := (A\otimes K_G)^G_+/\sim_G,$$
where Cuntz comparison is now witnessed by $G$-invariant sequences, that is, if
$B$ is a $G$-algebra and $a,b\in B^G_+$, then
	$$a\precsim_G b\qquad\text{if}\qquad \exists\seq x\subset B^G\ |\
	\norm{x_nbx_n^*-a}\to0,$$
where $B^G$ denotes the fixed point algebra of $B$ with respect to the action of
$G$. The binary operation is still derived from direct sum of positive elements,
that is
	$$[a]+[b] := [a\oplus b],$$
for any $[a],[b]\in\Cu^G(A)$.
\end{definition}

The approach of \cite{gs} is different, closer in spirit to the original
definition (see Definition \ref{def:equimvn}) of equivariant K-theory (cf.
\cite[Definition 2.4]{gs}). Finite dimensional representations of $G$ are
replaced by separable ones, i.e.~those representations $\mu$ of $G$ over a
separable Hilbert space $H_\mu$, and Cuntz classes of $G$-invariant positive elements
from the \Cs-algebras $K(H_\mu\otimes A)$ are now considered. Cuntz
comparison is then implemented by $G$-invariant elements from $K(H_\mu\otimes
A,H_\nu\otimes A)$, where $\nu$ is any other separable representation of $G$
(cf. \cite[Definition 2.6]{gs}).

With arguments similar to those of \S\ref{ssec:equiK} one can see that
indeed these two different approaches lead to the same equivariant Cuntz
semigroup for a continuous action of a compact group $G$ over a \Cs-algebra $A$.
The map $\Cu^G$ turns out to be a sequentially continuous functor from the
category of \Cs-algebras to the category $\Cu$.
In particular, this implies that, for every $G$-algebra $(A,G,\alpha)$, the equivariant Cuntz
semigroup $\Cu^G(A)$ is an object in $\Cu$ and, if $(B,G,\beta)$ is another
$G$-algebra, then every equivariant $*$-homomorphism $\pi:A\to B$ induces a
morphism $\Cu^G(\pi):\Cu^G(A)\to\Cu^G(B)$ in the category $\Cu$.

As with the ordinary Murray-von Neumann and Cuntz semigroups, there are similar
connections between the equivariant versions of these objects. Let
$(A,G,\alpha)$ be a $G$-algebra and $p\in (A\otimes K_G)^G$ a projection. The
map that sends the class of $p$ in $V^G(A)$ to the class of $p$ in $\Cu^G(A)$ is
a well defined semigroup homomorphism, as a consequence of the following
result, which generalises \cite[Lemma 2.18]{apt2009} to the equivariant
setting. Recall from Section \ref{ssec:equiK} that $\preceq_G$ denotes the equivariant Murray-von Neumann subequivalence relation between projections.

\begin{lemma} Let $(A,G,\alpha)$ be a $G$-algebra and let $p,q\in A^G$ be
$G$-invariant projections. Then $p\preceq_G q$ if and only if $p\precsim_G q$.
\end{lemma}
\begin{proof} Thanks to \cite[Proposition 2.5]{gs}, the same proof of
\cite[Lemma 2.18]{apt2009} applies almost verbatim by taking all the elements to
be $G$-invariant.
\end{proof}

The above result does not imply that Murray-von Neumann equivalence is equivalent to Cuntz equivalence on projections. However, as in the non-equivariant theory, there are special cases where the equivariant
Murray-von Neumann semigroup embeds into the equivariant Cuntz semigroup. A
stably finite $G$-algebra $(A,G,\alpha)$ is a $G$-algebra whose underlying
\Cs-algebra $A$ is stably finite. The following result is an equivariant
generalisation of \cite[Lemma 2.20]{apt2009}.

\begin{lemma}\label{lem:vinjects} Let $(A,G,\alpha)$ be a stably finite $G$-algebra. Then the
natural map $V^G(A)\to\Cu^G(A)$ is injective.
\end{lemma}
\begin{proof} Since the algebras are stably finite and $G$ is compact, the fixed point algebras and the crossed products are also stably finite. Hence, the same proof of \cite[Lemma 2.20]{apt2009} applies almost
verbatim by taking all the elements to be $G$-invariant.
\end{proof}

The \emph{completed} representation semiring
$\Cu(G)$, or simply the representation semiring, as defined in \cite[Definition
3.1]{gs}, is the semiring arising by considering separable representations $G$
rather than just the finite dimensional ones. We choose to include the word
\emph{complete} here because $\Cu(G)$ can be regarded as a $\sup$-completion of
the semiring $V^G(\CC)$, however we sometimes refrain from specifying this
explicitly. As in the case of K-theory, where $R(G)\cong
K_0^G(\CC)$, it turns out that there is an ordered semigroup isomorphism between
$\Cu(G)$ and $\Cu^G(\CC)$ \cite[Theorem 3.4]{gs}, which is then an object in the
category $\Cu$.

Let $(A,G,\alpha)$ be a $G$-algebra. Definition 3.10 and Theorem 3.11 of
\cite{gs} show that the equivariant Cuntz semigroup $\Cu^G(A)$ has a natural
$\Cu(G)$-semimodule structure and, as such, $\Cu^G(A)$ belongs to a subcategory
of $\Cu$, denoted $\Cu^G$ \cite[Definition 3.7]{gs}. As we are not
particularly interested in this category, we refer the reader to the already
cited work of Gardella and Santiago for more details. Here we limit ourself to
observing that, thanks to \cite[Theorem 3.11]{gs}, by equipping every
equivariant Cuntz semigroup with this $\Cu(G)$-semimodule structure, $\Cu^G$
becomes a functor from the category of $G$-algebras to the category $\Cu^G$.


\subsection{The Module Picture}

A module picture for the equivariant Cuntz semigroup is introduced in Section 4
of \cite{gs}. Part of the definitions we give here differ slightly from those
given in the cited work, but nonetheless lead to the same objects and results.

\begin{definition}[Equivariant Hilbert \Cs-module] Let $(A,G,\alpha)$ be a
$G$-algebra. An equivariant Hilbert $A$-module is a pair $(E,\rho)$ consisting
of a Hilbert $A$-module $E$ and a strongly continuous group homomorphism
$\rho:G\to U(E)$ such that
	\begin{enumerate}[i.]
		\item $\rho_g(a\xi) = \alpha_g(a)\rho_g(\xi),\quad\forall g,\in G, a\in
		A,\xi\in E$;
		\item $(\rho_g(\xi),\rho_g(\eta)) = \alpha_g((\xi,\eta)),\quad\forall g\in
		G, \xi,\eta\in E$.
	\end{enumerate}
\end{definition}

When the actions are understood, or there is no risk of confusion, we shall denote an
equivariant Hilbert $A$-module $(E,\rho)$ by its underlying Hilbert $A$-module
$E$ alone. Two equivariant Hilbert $A$-modules $(E,\rho)$ and $(F,\sigma)$ are
said to be equivariantly isomorphic (in symbols $E\cong_GF$) if there exists a
$G$-invariant unitary $u$, that is $u\circ\rho_g = \sigma_g\circ u$ for any
$g\in G$, in $B(E,F)^G$. An equivariant $A$-module $(Y,\eta)$ is said to be an
\emph{equivariant $A$-submodule} of $(E,\rho)$ if $Y$ is an $A$-submodule of $E$
which is stable under the action $\rho$ on $E$, i.e. $\rho_g(Y)\subset Y$ for
any $g\in G$, and $\eta$ coincides with the restriction of $\rho$ onto $Y$, that
is $\eta_g = \rho_g|_Y$ for any $g\in G$.

In order to define a Cuntz comparison of equivariant Hilbert \Cs-modules that
resembles the ordinary definition of \cite{cei} we need a notion of
\emph{equivariant compact containment}. This is done in \cite[Definition
4.10]{gs}, of which we give a slightly different version, that already
incorporates the comment that follows it in \cite{gs}, namely that the
contraction $a$ below can always be chosen to be $G$-invariant by a simple
averaging with respect to the normalised Haar measure on $G$.

\begin{definition} Let $(A,G,\alpha)$ be a $G$-algebra, $(E,\rho)$ an
equivariant $A$-module and $(F,\sigma)$ an equivariant $A$-submodule of $E$. We
say that $(F,\sigma)$ is equivariantly compactly contained in $(E,\rho)$ (in
symbols $F\cc_G E$) if there exists a $G$-invariant positive contraction $a\in
K(E)^G_+$ such that $a|_F=\id_F$.
\end{definition}

With the above definition, together with the notion of isomorphism of
equivariant Hilbert \Cs-modules we can now give a notion of Cuntz comparison for
these objects.

\begin{definition} Let $(A,G,\alpha)$ be a $G$-algebra and let $E$ and $F$ be
equivariant Hilbert $A$-modules. We say that $E$ is equivariantly
Cuntz-subequivalent to $F$ (in symbols $E\precsim_G F$) if
	$$\forall E'\cc_G E\qquad\exists F'\cc_G F\qquad|\qquad E'\cong_GF'.$$
\end{definition}

An equivariant Hilbert $A$-module $E$ is said to be countably generated if its
underlying Hilbert $A$-module is. Denoting the antisymmetrisation of
$\precsim_G$ by $\sim_G$, and the set of $G$-equivariant countably generated Hilbert $A$-modules by $H^G(A)$, we can then define the module picture of the Cuntz
semigroup of the $G$-algebra $A$ as the set of $\sim_G$-equivalence classes of
countably generated equivariant Hilbert $A$-modules, equipped with the binary
operation arising from the direct sum of modules, viz.
	$$\Cu^G_H(A) := H^G(A)/\sim_G.$$
A relation between the functors $\Cu^G$ and $\Cu^G_H$ can be established when $G$
is second countable. In this case, it turns out that there is a natural
isomorphism between $\Cu^G(A)$ and $\Cu^G_H(A)$ for any $G$-algebra $A$, that can
be taken, not only in the category \Cu, but even in the category $\Cu^G$ of
partially ordered $\Cu(G)$-semimodules briefly mentioned above (cf.
\cite[Proposition 4.13]{gs}).


\subsection{The Open Projection Picture}

We are now ready to introduce an open projection picture for the equivariant Cuntz semigroup as defined in this section.

\begin{definition} Let $A$ be a $G$-algebra. A $G$-invariant open projection is
an open projection in $(A^G)\dd$.
\end{definition}

The above definition entails that every $G$-invariant open projection is the
strong limit of an increasing sequence of positive elements from the fixed point
algebra.

\begin{lemma}\label{lem:modmean} If $(E,\rho)$ is an equivariant Hilbert
$A$-module of the form $\overline{aA}$ for some $a\in A_+$, then there exists
$\bar a\in A^G$ such that $E\cong_G \overline{\bar aA}$.
\end{lemma}
\begin{proof} Clearly $a\in E$. Since the map $g\mapsto \rho_g(a)$ is uniformly
continuous, for every $\eps>0$ there exists a neighbourhood $N$ of the identity
$e$ of the group $G$ such that $\norm{\rho_g(a)-a}<\eps$, for any $g\in N$.
Hence,
	\begin{align*}
		\int_G\rho_g(a)\de\mu(g) &\geq \int_N\rho_g(a)\de\mu(g)\\
			&\geq\int_N(a-\eps)_+\de\mu(g)\\
			&= \mu(N)(a-\eps)_+,
	\end{align*}
	with $\mu(N)>0$ by the regularity of the Haar measure $\mu$ on $G$.	By setting
		$$\bar a := \int_G\rho_g(a)\de\mu(g)$$
	one has $\bar a\in A_+$ and $p_{\bar a}\geq p_{(a-\eps)_+}$ for any $\eps>0$, so that
	$E\cong\overline{\bar aA}$, and $\rho_g(\bar a) = \bar a$ for any $g\in G$.
	For the inner product one has
	\begin{align*}
		(\rho_g(\bar ab),\rho_g(\bar ac)) &= \rho_g(\bar ab)^*\rho_g(\bar ac)\\
		&= \alpha_g(b)^*\bar a^2\alpha_g(c)\\
		&= \alpha_g(b^*\bar a^2c),\qquad\forall g\in G
	\end{align*}
	and by taking approximate units for $b$ and $c$ one then finds $\bar a^2 =
	\alpha_g(\bar a^2)$ for any $g\in G$, whence $\bar a\in A^G$.
\end{proof}
Let $a\in A^G$ be a $G$-invariant positive element and, like in the non-equivariant case, denote by $E_a$ the
equivariant Hilbert $A$-module generated by $(\overline{aA},\rho)$, where the
strongly continuous action $\rho$ is given by
	$$\rho_g(ab) := a\alpha_g(b)$$
for any $g\in G$. We give the following equivariant version of Blackadar
equivalence. The ordinary one is recovered by letting $G$ act trivially on $A$. Semantically this amounts to dropping the word ``equivariant'' and the superscripts $G$ from the following definition. Recall from page \pageref{eq:her} that $A_a$ denotes the hereditary \Cs-subalgebra of $A$ generated by the positive element $a\in A$.

\begin{definition} Let $A$ be a $G$-algebra. Two positive elements $a,b\in A^G$
are said to be \emph{equivariantly Blackadar equivalent}, in symbols
$a\sim_{G,s}b$, if there exists $x\in A^G$ such that $A_a=A_{x^*x}$ and
$A_b=A_{xx^*}$.
\end{definition}

For open projections $p,q\in (A^G)\dd$ we give the following equivariant version
of Peligrad-Zsid\'o equivalence. For the ordinary definition of this relation the same comment as above applies.

\begin{definition}\label{def:equipz} Let $A$ be a $G$-algebra. Two $G$-invariant open projections
$p,q\in (A^G)\dd$ are said to be \emph{equivariantly PZ equivalent}, in symbols
$p\gpze q$, if they are PZ equivalent with respect to $A^G$, i.e.~if there
exists a partial isometry $v\in (A^G)\dd$ such that
	$$p=v^*v,\qquad q=vv^*,$$
and
	$$v(A^G)_p\subset A^G,\qquad v^*(A^G)_q\subset A^G.$$
\end{definition}

A direct application of Kaplanski's density theorem shows that one might as well
use the notation $A^G_p$ to denote either $(A^G)_p$ or $(A_p)^G$, since both
these hereditary subalgebras coincide.

\begin{proposition} Let $A$ be a $G$-algebra and let $p$ be a $G$-invariant open
projection. Then $(A^G)_p=(A_p)^G$.
\end{proposition}
\begin{proof} The inclusion $(A^G)_p\subset(A_p)^G$ is obvious. By Kaplanski's
density theorem, every element $a\in(A_p)^G$ is a strong limit of a sequence of
elements $\seq a\subset(A_p)_{\norm a}$. For any vectors $\xi,\eta\in pH_u$,
where $H_u$ denotes the universal Hilbert space of $A$, one has the estimate
	$$|(\xi,\alpha_g(a_n)\eta)|\leq\norm\xi\norm\eta\norm a,\qquad\forall
	n\in\IN,g\in G.$$
Therefore, by Lebesgue's dominated convergence theorem one can interchange the
order of limit and integral in
	$$a = \int_G\SOT\lim_{n\to\infty}p\alpha_g(a_n)p\ \de\mu(g)$$
to get
	$$a = \SOT\lim_{n\to\infty}p\left(\int_G\alpha_g(a_n)\de\mu(g)\right)p$$
with the average $\int_G\alpha_g(a_n)\de\mu(g)$ belonging to $A^G$ for any
$n\in\IN$. Hence, $a\in (A^G)_p$.
\end{proof}

The result that follows can be regarded as an equivariant extension of
Proposition 4.3 of \cite{ORT}.

\begin{proposition}\label{prop:equiequi} Let $A$ be a $G$-algebra and let $a$
and $b$ be $G$-invariant positive elements of $A$. The following are equivalent:
	\begin{enumerate}[i.]
		\item $a\sim_{G,s}b$;
		\item $E_a$ and $E_b$ are equivariantly isomorphic;
		\item $\exists x\in A^G$ such that $E_a=E_{x^*x}$ and $E_b = E_{xx^*}$;
		\item $p_a\sim_{G,PZ}p_b$.
	\end{enumerate}
\end{proposition}
\begin{proof} $i.\Rightarrow iv.$ As a direct consequence of \cite[Theorem
1.4]{pz} one has that $p_{x^*x}\sim_{G,PZ}p_{xx^*}$, since this is true for
$p_{x^*x}\sim_{PZ}p_{xx^*}$ in $A^G$. Furthermore, $A_a = A_b$, with $a,b\in
A^G$, implies that $p_a=p_b$, with $p_a$ and $p_b$ in $(A^G)\dd$.

$iv.\Rightarrow i.$ By the arguments of \cite[Proposition 4.3]{ORT}, one sees
that, if $v$ denotes the partial isometry that witnesses the $PZ$ equivalence of
$p_a$ and $p_b$, then $vav^*\in A^G$ has the same support projection of $b$,
i.e. $p_b$, in $A^G$.

$ii.\Rightarrow iii.$ Let $u$ be the map that implements the equivariant
isomorphism and set $x:=ua$. Then $E_{xx^*}=\overline{xA} = \overline{uaA} =
E_b$ and
	$$\sigma_g(x) = (\sigma_g\circ u\circ\rho_g^{-1}\circ\rho_g)(a) = u\rho_g(a) =
	ua = x,\qquad\forall g\in G,$$
	therefore $x\in A^G$. Furthermore, $x^*x = a^2$ since $u$ is isometric, so
	that $E_a = E_{x^*x}$.
	
	$iii.\Rightarrow ii.$ Let $x=v|x|$ be the polar decomposition of $x$, with
	$v\in (A^G)\dd$. Then
	\begin{align*}
		v\rho_g(|x|b) &= v|x|\alpha_g(b)\\
			&= |x^*|v\alpha_g(b)\\
			&= \sigma_g(|x^*|v)\alpha_g(xb)\\
			&= \sigma_g(|x^*|vb)\\
			&= \sigma_g(v|x|b)
	\end{align*}
	for any $b\in A$, whence $v\in B(E_{x^*x},E_{xx^*})^G$ is the sought
	equivariant isomorphism.
	
	$i.\Leftrightarrow iii.$ This is a restatement of the definitions involved and
	based on the one-to-one correspondence between hereditary subalgebras and
	right ideals.
\end{proof}

The following is an equivariant version of the compact containment relation for
open projections.

\begin{definition}\label{def:equicc} Let $A$ be a $G$-algebra, and let $p,q\in (A^G)\dd$ be $G$-invariant open projections. We say that $q$ is compactly contained in $p$ (in symbols $q\cc_G p$) if there exists $e\in A_p^G$ such that $\bar qe = \bar q$, where $\bar q$ denotes the closure of $q$.
\end{definition}

With both Definition \ref{def:equipz} and Definition \ref{def:equicc} one can define the Cuntz comparison of two $G$-invariant open projections in the usual way of \cite{cei} and \cite{ORT}.

\begin{definition} Let $(A,G,\alpha)$ be a $G$-algebra and let $p,q$ be $G$-invariant open projections from $(A^G)\dd$. We shall say that $p$ is equivariantly Cuntz subequivalent to $q$ ($p\precsim_Gq$ in symbols) if
	$$\forall p'\cc_G p\qquad\exists q'\cc_G q\qquad|\qquad p'\gpze q'.$$ 
\end{definition}
Hence, two $G$-invariant open projections $p$ and $q$ are said to be Cuntz equivalent if both $p\precsim_Gq$ and $q\precsim_G p$ hold.

The proposition below can be regarded as an equivariant extension of part of the
results established in \cite[Proposition 4.10]{ORT}.

\begin{proposition}\label{prop:equicc} Let $A$ be a $G$-algebra and let $a,b$ be
$G$-invariant positive elements. Then $E_a\cc_G E_b$ if and only if $p_a\cc_G
p_b$.
\end{proposition}
\begin{proof} Identify $K(E_b)$ with $A_b$ and observe that the rank-1 operator
$\theta_{bd,bc}$ is sent to the element $bdc^*b$ for any $c,d\in A$. Hence, the action $\Ad_{\rho_g}$ on $K(E_b)$ coincides with the action of $\alpha_g$ on $A_b$.
Therefore, if $e\in K(E_b)^G$ is such that $e|_{E_a}=\id_{E_a}$, then $e\in
A_b^G$ is such that $\overline{p_a}e=\overline{p_a}$.
\end{proof}

Propositions \ref{prop:equiequi} and \ref{prop:equicc} can now be used to
\emph{translate} the module picture of the previous section into the open
projection picture for the equivariant Cuntz semigroup.

\begin{theorem} Let $G$ be a second countable compact group. Then $\Cu^G(A)\cong
P_o(((A\otimes K_G)^G)\dd)/\sim_G$.
\end{theorem}
\begin{proof} By Proposition \ref{prop:equiequi}, equivariant isomorphism of
modules coincides with equivariant PZ equivalence of the corresponding
$G$-invariant open projections, while by Proposition \ref{prop:equicc} compact containment
of equivariant modules corresponds to compact containment of $G$-invariant open
projections. Hence, it is enough to show that there exists a bijection between
	$$E_a^{\cc_G} := \{X\ |\ X\cc_G E_a\}$$
and
	$$p_a^{\cc_G} := \{p\ |\ p\cc_G p_a\}$$
for any positive element $a\in (A\otimes K_G)^G$. To this end, suppose that $X\cc_GE_a$. Since
$A\otimes K_G$ is a stable \Cs-algebra, there exists $a'\in (A\otimes K_G)_+$ such
that $X=\overline{a'(A\otimes K_G)}$, and by Lemma \ref{lem:modmean} one can
assume that $a'$ is $G$-invariant. By Proposition \ref{prop:equicc},
$E_{a'}\cc_GE_a$ is equivalent to $p_{a'}\cc_G p_a$, so that one can associate
the $G$-invariant projection $p_{a'}$ to the equivariant module $X$. To see that
this correspondence is well-defined and independent from the choice of $a'$,
observe that, if $a''\in A\otimes K_G$ is another $G$-invariant positive element
such that $X=\overline{a''(A\otimes K_G)}$, then the hereditary subalgebra
generated by $a''$ is the same as that generated by $a'$ and therefore
$p_{a''}=p_{a'}$. Conversely, for every $p\cc_G p_a$ there exists $a'\in (A\otimes K_G)^G$ such that $p=p_{a'}$, and by Proposition \ref{prop:equicc} again this
implies that $E_{a'}\cc_G E_a$. Any other choice of a positive element that gives
the same open projection leads to the same hereditary subalgabra and hence to
the same module, whence it follows that the correspondence $p\mapsto E_{a'}$ is
well-defined and independent from the choice of $a'$. It is now immediate to
verify that this correspondence is the inverse of the one above and therefore it
provides a bijection between $p_a^{\cc_G}$ and $E_a^{\cc_G}$.
\end{proof}

	\section{\label{sec:classification}Classification of Actions}

In this section we show how to use the equivariant extension of the bivariant Cuntz semigroup defined in this paper in order to retrieve some classification results for certain actions by compact groups over certain \Cs-algebras. In particular we show how to recover the classical result of Handelman and Rossmann \cite{hr}, and the more recent one of Gardella and Santiago \cite{gs}.

With the aim of capturing the right notion of equivalence, which incorporates the
scale conditions of part (2) of \cite[Theorem 8.4]{gs}, we give the following
definition as the equivariant analogue of the \emph{strictly invertible} elements defined in \cite{bcs}.

\begin{definition} Let $A$ and $B$ be $G$-algebras. An
element $\Phi\in\Cu^G(A,B)$ is said to be \emph{strictly invertible} if there
exist equivariant c.p.c.~order zero maps $\phi:A\to B$ and $\psi:B\to A$ such
that
	\begin{enumerate}[i.]
		\item $[\phi\otimes\id_{K_G}] = \Phi$;
		\item $\psi\circ\phi\sim_G\id_A$ and $\phi\circ\psi\sim_G\id_B$.
	\end{enumerate}
\end{definition}

As in the case of the non-equivariant theory of the bivariant Cuntz semigroup of \cite{bcs}, one can regard strictly invertible elements as invertible elements of $\Cu^G(A,B)$, with the obvious meaning of invertibility, that come from the \emph{scale} of $\Cu^G(A,B)$, where the latter is defined as follows.

\begin{definition} Let $A$ and $B$ be $G$-algebras. The scale of $\Cu^G(A,B)$ is the set of classes
	$$\Sigma(\Cu^G(A,B)) := \{[\phi\otimes\id_{K_G}]\in\Cu^G(A,B)\ |\ \phi:A\to B \text{ equiv.~c.p.c.~order zero}\}.$$
\end{definition}

An isomorphism criterion for crossed products follows from the following result about the map $j_G$ of Section \ref{ssec:cp}.

\begin{proposition} Let $A$ and $B$ be $G$-algebras. If $\Phi\in\Cu^G(A,B)$ is a strictly invertible element then so is $j_G(\Phi)\in\Cu(A\rtimes G, B\rtimes G)$, in the sense of \cite{bcs}.
\end{proposition}
\begin{proof} Observe that $(\phi\otimes \id_{K_G})_\rtimes$ can be identified with $\phi_\rtimes\otimes\id_K$ for any equivariant c.p.c.~order zero map $\phi:A\to B$. Since $\Phi$ is strictly invertible in $\Cu^G(A,B)$, there are equivariant c.p.c.~order zero maps $\phi:A\to B$ and $\psi:B\to A$ such that $[\phi\otimes\id_{K_G}] = \Phi$ and $\psi\circ\psi\sim_G\id_A$, $\phi\circ\psi\sim_G\id_B$. From Theorem \ref{th:jG} it follows that
	$$j_G(\Phi) = [(\phi\otimes\id_{K_G})_\rtimes]$$
whereas from Proposition \ref{prop:rtimes} one has that
	$$\psi_\rtimes\circ\phi_\rtimes \sim \id_{A\rtimes G}$$
and
	$$\phi_\rtimes\circ\psi_\rtimes \sim \id_{B\rtimes G}.$$
Hence $j_G(\Phi)\in\Cu(A\rtimes G, B\rtimes G)$ is a strictly invertible element.
\end{proof}

As an immediate consequence of the above proposition we then have the following isomorphism criterion for crossed products by finite groups.

\begin{theorem} Let $A$ and $B$ be unital and stably finite $G$-algebras, $G$ finite. If there is a strictly invertible element in $\Cu^G(A,B)$ then $A\rtimes G$ and $B\rtimes G$ are isomorphic.
\end{theorem}
\begin{proof} If $\Phi\in\Cu^G(A,B)$ is a strictly invertible element, then $j_G(\Phi)$ is strictly invertible in $\Cu(A\rtimes G, B\rtimes G)$. Furthermore, $A\rtimes G$ and $B\rtimes G$ are unital and stably finite, and therefore the classification theorem of \cite{bcs} applies.
\end{proof}

For any pair of $G$-algebras $A$ and $B$, it is easy to see that there is a well-defined map $\sigma_G$ from the scale of $\Cu^G(A,B)$ to the scale of $\Cu(A^G,B^G)$, which is given by
	$$\sigma_G([\phi\otimes\id_{K_G}]) := [\phi|_{A^G}\otimes\id_K].$$
In particular, it follows that any strictly invertible element of $\Cu^G(A,B)$ yields a strictly invertible element of $\Cu(A^G,B^G)$. Hence

\begin{theorem} Let $A$ and $B$ be unital and stably finite $G$-algebras. If there is a strictly invertible element in $\Cu^G(A,B)$, then the fixed point algebras $A^G$ and $B^G$ are isomorphic.
\end{theorem}
\begin{proof} By the above considerations, if $\Phi\in\Cu^G(A,B)$ is strictly invertible, then so is $\sigma_G(\Phi)\in\Cu(A^G,B^G)$. Since $A$ and $B$ are unital, $A^G$ and $B^G$ are unital and stably finite, and the classification theorem of \cite{bcs} applies.
\end{proof}

As in the standard theory of the bivariant Cuntz semigroup, we have the following
result for the equivariant setting.

\begin{proposition}\label{prop:cpctohom} Let $A,B$ be unital and
stably finite $G$-algebras. If $\phi:A\to B$ and $\psi:B\to A$ are two equivariant c.p.c.
order zero maps such that $\psi\circ\phi\sim_G \id_A$ and $\phi\circ\psi\sim_G\id_B$
then there are equivariant unital $*$-homomorphisms $\pi_\phi:A\to B$ and $\pi_\psi:B\to A$
such that
	\begin{enumerate}[i.]
		\item $[\pi_\phi] = [\phi]$ and $[\pi_\psi] = [\psi]$;
		\item $\pi_\psi\circ\pi_\phi \sim_G \id_A$ and
		$\pi_\phi\circ\pi_\psi \sim_G \id_B$.
	\end{enumerate}
\end{proposition}
\begin{proof} By Theorem \ref{thm:equiwz} we can find $G$-invariant positive elements $h_\phi,h_\psi$ and equivariant $*$-homomorphisms
$\pi_\phi,\pi_\psi$ such that $\phi = h_\phi\pi_\phi$ and $\psi =
h_\psi\pi_\psi$. Evaluating on the unit of $A$ and $B$ respectively we get
		$$h_\psi^{\frac12}\pi_\psi(h_\phi)h_\psi^{\frac12}\sim_G 1_A
		\qquad\text{and}\qquad
		h_\phi^{\frac12}\pi_\phi(h_\psi)h_\phi^{\frac12}\sim_G 1_B,$$
	where by $\sim_G$ we mean that the sequences that witnesses the Cuntz equivalence are taken from the fixed point algebras. Hence, there exists $\seq x\subset A^G$ such that
	$x_nh_\psi^{\frac12}\pi_\psi(h_\phi)h_\psi^{\frac12}x_n^*$ converges to $1_A$,
	and therefore $x_nh_\psi^{\frac12}\pi_\psi(h_\phi)h_\psi^{\frac12}x_n^*$ is
	eventually invertible. From this we conclude that there exists $c\in A$ such that
		$$x_nh_\psi^{\frac12}\pi_\psi(h_\phi)h_\psi^{\frac12}x_n^*c = 1_A$$
	for sufficiently large values of $n$, which shows that $x_n$ is right
	invertible. Since $A$ is stably finite, it follows that the sequence $\seq x$
	is eventually invertible, and therefore
		$$h_\psi^{\frac12}\pi_\psi(h_\phi)h_\psi^{\frac12}x_n^*cx_n = 1_A,$$
	which shows that $h_\psi$ is also right invertible, hence invertible.
	Similarly, one also deduces the invertibility of $h_\phi$, and so
	$\pi_\phi$ and $\pi_\psi$ satisfy (i) and (ii). Now set $p=\pi_\phi(1_A)$
	and $q = \pi_\psi(1_B)$. Since $\pi_\psi(p)\sim_G 1_A$ and
	$\pi_\phi(q)\sim_G 1_B$, stably finiteness of $A$ and $B$ implies
	$\pi_\phi(q) = 1_B$ and $\pi_\psi(p) = 1_A$. Now $1_A-\pi_\psi(q)$ is a
	positive element in $A^G$, but
		$$\pi_\phi(1_A-\pi_\psi(q)) = p - 1_B \leq 0,$$
	which is possible only if $p = 1_B$. Similarly, one finds that $q=1_A$, and
	therefore $\pi_\phi$ and $\pi_\psi$ are unital. Finally, observe that the invertibility of $h_\phi$ and $h_\psi$ implies that the ranges of $\pi_\phi$ and $\pi_\psi$ are $B$ and $A$ respectively.
\end{proof}

\begin{theorem}\label{thm:Gsinvtoisowithscale} Let $A$ and
$B$ be unital and stably finite $G$-algebras. Every strictly invertible
element $\Phi\in\Cu^G(A,B)$ induces a $\Cu(G)$-semimodule
isomorphism $\rho:\Cu^G(A)\to\Cu^G(B)$ such that $\rho([1_A])=[1_B]$ and
$\rho([1_A\otimes e_G])=[1_B\otimes e_G]$.
\end{theorem}
\begin{proof} Thanks to Proposition \ref{prop:cpctohom}, if $\Phi\in\Cu^G(A,B)$ is a strictly invertible
element, there are equivariant c.p.c.~order zero maps $\phi:A\to B$ and
$\psi:B\to A$ such that $\psi\circ\phi\sim_G\id_A$ and
$\phi\circ\psi\sim_G\id_B$, which can then be replaced by their support
$*$-homomorphisms $\pi_\phi$ and $\pi_\psi$ respectively. Then
$\rho:=\Cu^G(\pi_\phi)$ is a $\Cu(G)$-semimodule isomorphism that satisfies
$\rho([1_A]) = [1_B]$ and clearly sends the constant function $1_A(g) = 1_A$ to
$1_B(g) = 1_B$, whence $\rho([1_A\otimes e_G]) = [1_B\otimes e_G]$.
\end{proof}

\begin{definition}\label{def:locrep} Let $(A,G,\alpha)$ be a $G$-algebra. The
action $\alpha$ on $A$ is said to be \emph{representable} if there exists a
group homomorphism $u:G\to U(\mathcal M(A))$ such that $\alpha_g = \Ad(u_g)$ for
any $g\in G$. The action $\alpha$ is said to be \emph{locally representable} if
there exists an increasing sequence $\seq A$ of $\alpha$-invariant
\Cs-subalgebras of $A$ such that $\bigcup_{n\in\IN}A_n$ is dense in $A$ and
$\alpha|_{A_n}$ is representable for every $n\in\IN$.
\end{definition}

In the following corollary we borrow the definition of the class of algebras $\mathbf R$ and that of locally representable actions from \cite{gs} (see the discussion that precedes \cite[Theorem 8.4]{gs}).

\begin{corollary}\label{cor:bgs} Let $G$ be a finite Abelian group and let
$(A,G,\alpha)$ and $(B,G,\beta)$ be unital $G$-algebras in the class $\mathbf R$
with locally representable actions $\alpha$ and $\beta$ along given inductive sequences for $A$ and $B$ respectively, that lie in the class $\mathbf R$. Then $(A,G,\alpha)$ and $(B,G,\beta)$ are equivariantly
isomorphic if and only if there is a strictly invertible element in
$\Cu^G(A,B)$.
\end{corollary}
\begin{proof} It follows directly from the above theorem, together with the classification results of \cite{gs}.
\end{proof}

Locally representable actions for the larger class of compact groups have been
considered by Handelman and Rossmann. Their definition of local representability
is restricted to AF algebras, and it is assumed that an action $\alpha$ over an
AF algebra $A$ is locally representable if it is representable along a given
inductive sequence of \emph{finite dimensional} \Cs-algebras whose limit is $A$. We shall say that a
$G$-algebra $(A,G,\alpha)$ is AF if the underlying \Cs-algebra $A$ is. Their
main classification result \cite[Theorem III.1]{hr}, for the purposes of this
paper, can be stated in the following way.

\begin{theorem}[Handelman-Rossmann]\label{thm:hr} Let $G$ be a compact group and
let $(A,G,\alpha)$ and $(B,G,\beta)$ be unital AF $G$-algebras, with $\alpha$
and $\beta$ locally representable actions along given inductive sequences for
$A$ and $B$ respectively. Then $A$ and $B$ are equivariantly isomorphic if and
only if there exists a $V^G(\CC)$-semimodule isomorphism $\rho:V^G(A)\to V^G(B)$
such that $\rho([1_A]) = [1_B]$.
\end{theorem}

This classification result can be recovered from the equivariant
bivariant Cuntz semigroup as a corollary to Theorem
\ref{thm:Gsinvtoisowithscale}, as it is now shown.

\begin{corollary}\label{cor:bhr} Let $G$ be a compact group and let
$(A,G,\alpha)$ and $(B,G,\beta)$ be unital AF $G$-algebras, with $\alpha$ and
$\beta$ locally representable actions along given inductive sequences for $A$
and $B$ respectively. Then $A$ and $B$ are equivariantly isomorphic if and only
if there exists a strictly invertible element in $\Cu^G(A,B)$.
\end{corollary}
\begin{proof} Recall that, by Lemma \ref{lem:vinjects}, $V^G(A)$ injects in $\Cu^G(A)$ for any stably finite \Cs-algebra $A$. By Theorem \ref{thm:Gsinvtoisowithscale}, every strictly
invertible element is represented by an equivariant $*$-homomorphism, which maps $G$-invariant projections to $G$-invariant projections, and therefore induces a $V^G(\CC)$-semimodule
homomorphism between $V^G(A)$ and $V^G(B)$ that satisfies all the hypotheses
of Theorem \ref{thm:hr}.
\end{proof}

	\bibliographystyle{hplain}
	\bibliography{refs}
	
\end{document}